\pdfoutput=1
\documentclass[11pt]{amsart}
\usepackage[T1]{fontenc}
\usepackage[utf8]{inputenc}
\usepackage{amssymb,amsmath,amsthm,mathrsfs,array,graphicx,bbm,enumerate,wasysym,dsfont,xcolor}
\usepackage[most]{tcolorbox}

\usepackage[all]{xy}
\usepackage{bbm}
\usepackage{fullpage}
\usepackage{float}
\usepackage{verbatim}
\usepackage[dvipsnames]{xcolor}
\usepackage{hyperref}

\newtheorem{thm}{Theorem}[section]
\newtheorem{lemma}[thm]{Lemma}
\newtheorem{rem}[thm]{Remark}
\newtheorem{defn}[thm]{Definition}
\newtheorem{prop}[thm]{Proposition}
\newtheorem{cor}[thm]{Corollary}
\newtheorem{claim}[thm]{Claim}
\newtheorem{remark}[thm]{Remark}

\numberwithin{equation}{section}

\newcommand{\Q}{\mathbb Q}
\newcommand{\Z}{\mathbb Z}
\newcommand{\R}{\mathbb R}
\newcommand{\C}{\mathbb C}

\newcommand{\N}{\mathbb N}

\newcommand{\E}{\mathbb E}
\newcommand{\EE}{\mathcal{E}}

\renewcommand{\P}{\mathbb P}
\renewcommand{\1}{\mathbf 1}

\newcommand{\B}{\mathcal B}

\newcommand{\M}{\operatorname{M}}

\renewcommand{\epsilon}{\varepsilon}

\newcommand{\Tt}{\tau}

\newcommand{\ddd}{\delta}

\newcommand{\Cov}{\mathrm{cov}}
\newcommand{\ET}{\mathrm{ET}}
\newcommand{\Var}{\mathrm{var}}

\newcommand{\Co}{\textsuperscript{o}}
\newcommand{\hs}{\hspace{0.1cm}}
\newcommand{\g}{\gamma}
\newcommand{\vp}{\varPhi}
\newcommand{\vpn}{\varPhi_{r_n}}

\newcommand{\pe}[1]{\Psi_{#1}}
\newcommand{\ipp}[2]{\left\langle #1, #2 \right\rangle}

\newcommand{\ipn}[2]{\left\langle #1, #2 \right\rangle_{\nabla}}

\newcommand{\Os}[2]{\overset{#1}{#2}}
\newcommand{\Hs}[2]{\hyperref[#1]{#2}}
\newcommand{\vpe}[1]{\Phi_{#1}}
\newcommand{\mbb}[1]{\mathbb{#1}}
\newcommand{\mcl}[1]{\mathcal{#1}}

\newcommand{\Gr}{G^{D}}
\newcommand{\goesto}{\to}

\newcommand{\f}[1]{f(z_{#1})}

\let \ln \log

\begin{document}
\title{Thick points under Gaussian free field dynamics}
\author{ Felipe Espinosa Vergara$^\dagger$}
\address {$\dagger$ Fachbereich Mathematik und Informatik, Universität Münster, Einsteinstraße 62, Münster 48149, Germany}
\author{Avelio Sepúlveda$^*$}
\address{$^*$ Departamento de Ingenier\'ia Matem\'atica and Centro de Modelamiento Matem\'atico (IRL-CNRS 2807)\\
  Universidad de Chile\\
  Av. Beauchef 851, Torre Norte, Piso 5\\
  Santiago\\
  Chile }
\email{}

\maketitle
\begin{abstract}
We investigate the evolution of thick points under two natural dynamics for the Gaussian free field (GFF) in dimension 2. The first dynamic we analyze is the Ornstein-Uhlenbeck GFF. We prove that, simultaneously for all points, the evolution of their thickness is continuous. Additionally, we characterize all deterministic functions $f: \R\to \R$ such that there are points whose thickness function is $f$. 

The second dynamic we study is the stationary solution of the additive stochastic heat equation. In this case, the thickness of points is not continuous. Moreover, this rougher dynamic generates super-thick points, namely points with thickness greater than $2$. As a function of $\gamma > 2$, we identify infinitely many phase transitions corresponding to the existence of exceptional times where at least $N$ points are $\gamma$-thick. These phase transitions, occurring at $\gamma^2 = 8, 6, 16/3, \dots$, converge to $4$ as $N \to \infty$. Mapping to the critical FK-model parameter via $q=4\cos^2(4\pi/\gamma^2)$, these critical values correspond to the Beraha numbers, which are precisely the points at which the CFT for critical FK percolation should be minimal.
\end{abstract}

\section{Introduction}

In the past two decades, there has been significant interest in the study of the geometry of the two-dimensional Gaussian free field (GFF). This interest stems from the fact that the GFF is a powerful tool for understanding important conformally invariant objects and constructing their scaling limits. The first class of such objects includes those where the GFF itself is the limiting object, as seen in the fluctuations of the height function in domino tilings \cite{ken}, the characteristic polynomials of random matrices \cite{RV}, the Integer-Valued GFF \cite{BPR1,BPR2}, and the height function of both the six vertex model \cite{DKLM} and of the double current model \cite{DLQ}. Other examples appear as functionals of the GFF: its exponential plays a key role in constructing the scaling limit of random planar maps \cite{MS21,GMS,HS}, while its trigonometric functions emerge in the scaling limits of the XOR-Ising model \cite{DLQ,AS26,ALH} and, conjecturally, the Ashkin-Teller model.

The overarching goal of this article is to lay the groundwork for understanding the geometry of the two-dimensional GFF under natural time evolutions. We approach this through two distinct dynamics. We first study the Ornstein-Uhlenbeck GFF (OU-GFF), an Ornstein-Uhlenbeck process in the Sobolev space $H_0^1$; this process was introduced by Sheffield in Remark 2.16 of \cite{dimGFF}; however, to our knowledge, this is the first time its geometry has been studied. We then turn to the stationary solution of the additive stochastic heat equation (SHEF). Despite the SHEF acting as a basic model in the study of SPDEs, the temporal evolution of its geometric structures has remained unaddressed.

In the static case, despite the fact that the two-dimensional GFF itself is not an $L^2$-function but rather a Schwartz distribution, its geometric properties are well understood. Notable features include its level lines \cite{SS,WW,PW}, flow lines \cite{MS1,MS2,MS3,MS4,aru2015kpz}, level sets  \cite{ASW,AS18,ALS1,ALS2}, excursion decomposition \cite{ALS4} and extreme values \cite{HMP}. In this paper, we focus on the latter, specifically the study of thick points. This choice is motivated by two key factors: technically, thick points are the easiest to study, and, more importantly, they play a crucial role in Liouville Quantum Gravity (LQG) theory, as they are the points that carry the natural mass of the associated metric space.

The main results of this article concern the behavior of thick points in each of the two dynamics. For the OU-GFF, we study the thickness function $\gamma_x(t)$, which gives the thickness of the OU-GFF at a spatial point $x$ and time $t$. We show that, almost surely, for any $x$, the function $t\mapsto \gamma_x(t)$ is continuous where it is defined (see Figure \ref{f.OUGFF} to follow the thickness function of the original maximum). Furthermore, we characterize all deterministic functions that can be followed by the thickness of a point $x$.

\begin{figure}
    \centering
    \includegraphics[width=0.7\linewidth]{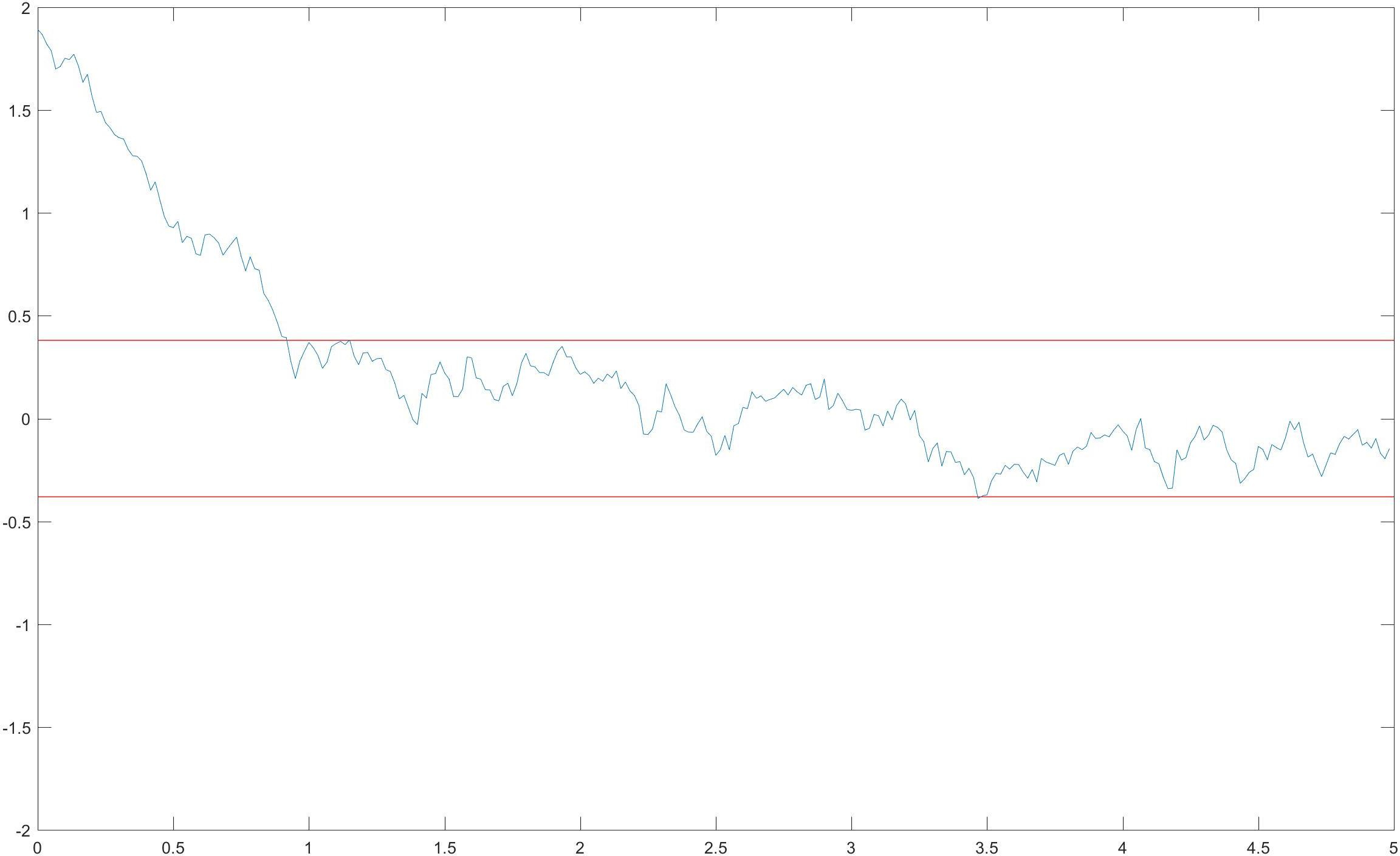}
    \caption{Simulation of the evolution of the thickness of the point $x$ that attains the maximum at time $0$ of the OU-GFF. The red lines indicate the value $1/\sqrt{\log(N)}$ which is the typical thickness of points in this simulation, where $N=2000$. For videos of these simulations, see \cite{SimOUGFF}.}
    \label{f.OUGFF}
\end{figure}

In the case of the SHEF, we show that for a given point $x\in D$, a.s. the thickness function is not continuous. However, a new phenomenon appears. To describe it, we introduce the notion of super-thick points: points $(x,t)\in D\times \R$ where the SHEF at time $t$ is $\gamma$-thick for some $\gamma>2$ (see Figure \ref{f.SHEF} to follow the maximum thickness at a given time). We show that super-thick points exist for all values of $\gamma <\sqrt{8}$. Additionally, for any $N\in \N$, we study exceptional times $t$ where there are distinct points $\{x_1,\cdots,x_N\}$ such that the SHEF is $\gamma$-thick at $(x_i,t)$. We show that for any such $N$, there is a critical value $\gamma_N>2$, such that for $\gamma<\gamma_N$, there exist $(N,\gamma)$-exceptional times, while for $\gamma>\gamma_N$, no such points exist. Furthermore, we show that the sequence of critical values satisfies $(\gamma_N^c)^2=\frac{4}{N}+4$, thus establishing a numerical connection with the Beraha numbers, and hence with special values of the Potts model parameter (see Section \ref{sss.Beraha} for a short discussion on this coincidence).

\begin{figure}
    \centering
    \includegraphics[width=0.7\linewidth]{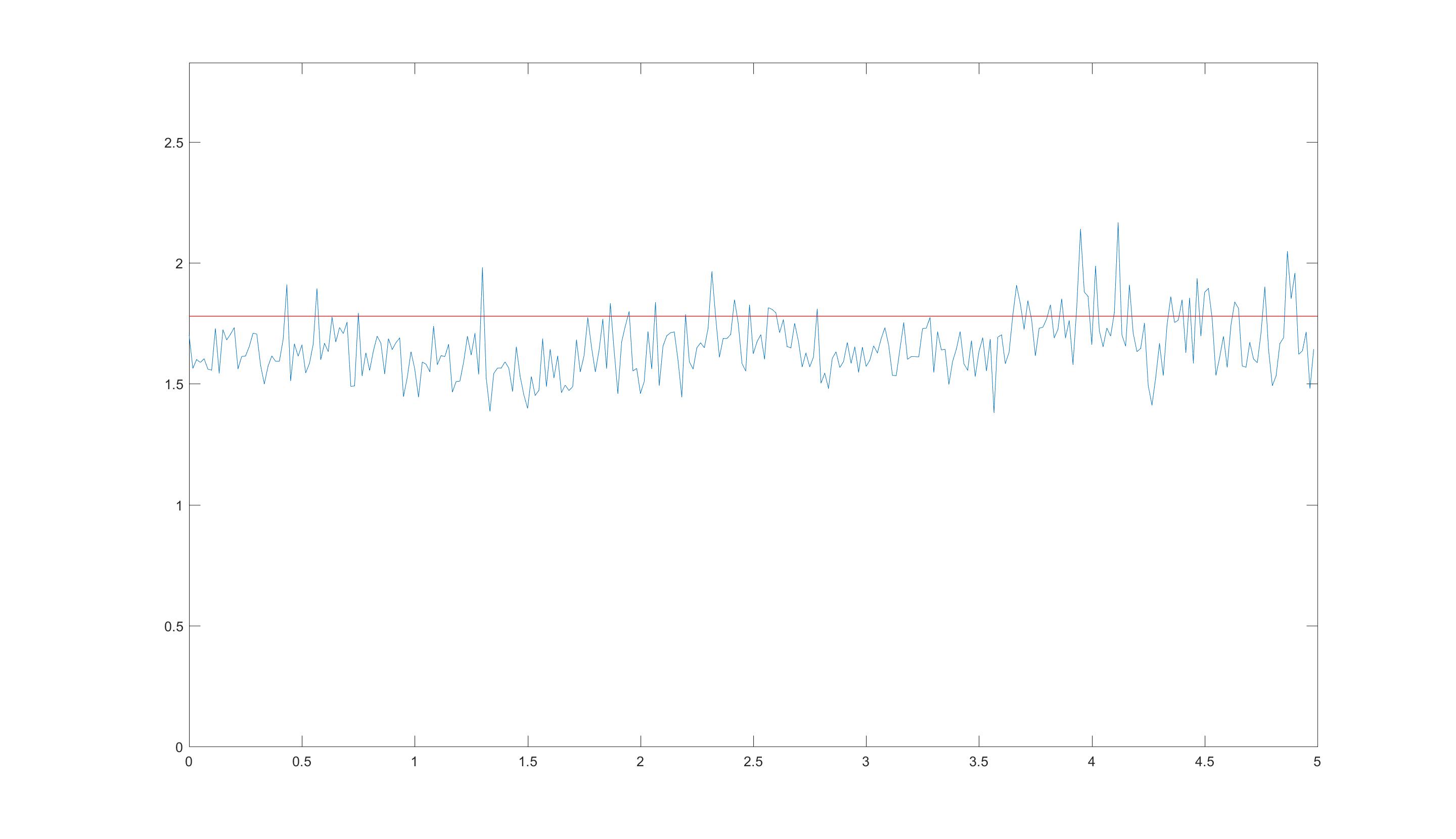}
    \caption{Simulation of the evolution of the maximum thickness of the SHEF. The red lines indicate the value $2-\frac{3 }{4 }\frac{\log(\log(N)) }{\log(N) }$ which is the typical value for the maximum of a GFF in the grid $N\times N$, here $N=1000$. Note that as the size is small, the log-log correction is actually macroscopic.  For videos of these simulations, see \cite{SimSHEF}.}
    \label{f.SHEF}
\end{figure}

\subsection{Results} Let $D\subseteq \C$ be a bounded domain and let $(e_k,\lambda_k)_{k\in \N}$ be the pairs of eigenfunctions and eigenvalues of minus the Dirichlet Laplacian on $D$ where $e_k$ has an $L^2$ norm\footnote{We choose this normalisation so that the correlation of the GFF is $G_D(x,y)\sim -\log(\|x-y\|)$. 
} equal to $\sqrt{2\pi}/\sqrt{\lambda_k}$.
In this context, we define the GFF as the distribution given by
\begin{align*}
\sum_k \alpha_k e_k ,
\end{align*}
where $(\alpha_k)_{k\in \N}$ is an i.i.d. sequence of standard Gaussian random variables. In particular, the GFF is the standard Gaussian in $H_0^1(D)$ with inner product
\begin{align}\label{e.normH01}
\ipn{f}{g}= \frac{1 }{2\pi} \int_D\nabla f\cdot \nabla g dx
\end{align}

\subsubsection{Ornstein-Uhlenbeck GFF}Let us first introduce the results regarding the OU-GFF. This is the field constructed as
\begin{align*}
\Phi(t,\cdot)= \sum_k \alpha_k(t) e_k(\cdot) ,
\end{align*}
where $(\alpha_k(t))_{n\in \N}$ are i.i.d. Ornstein-Uhlenbeck processes in $\R$. More precisely, each $\alpha_k$ solves the following SDE.
\begin{align*}
d\alpha_k(t) = -\frac{1 }{2 }\alpha_k(t)dt +dB^{(k)}(t),
\end{align*}
where $(B^{(k)})_k$ are independent Brownian motions. 

In this context, we define the thickness function at time $t$ and at point $x$ as
\begin{align*}
\gamma_x^+(t) = \limsup_{\epsilon \to 0}\frac{\Phi_\epsilon(x,t) }{\log(1/\epsilon) } \ \text{ and } \ \gamma_x(t)= \lim_{\epsilon\to 0} \frac{\Phi_\epsilon(x,t) }{\log(1/\epsilon)},
\end{align*}
where the circle average $\Phi_{\epsilon}(x,t)$ is defined as $\langle \Phi(t,\cdot),\mu_\epsilon^x \rangle$, and $\mu_{\epsilon}^x$ is the uniform measure on the spatial circle $\partial B(x,\epsilon)\subseteq \C$. Note that the function $\gamma$ is not defined for all space-time points, however, the function $\gamma^+$ is. Also recall that for a GFF the maximum value of $\gamma_{x}(0)$ is $2$ \cite{HMP}.

Our first result concerns the continuity of the function $\gamma^+_x(\cdot)$ and, consequently, that of the function $\gamma_x(\cdot)$ where it is defined.
\begin{prop}\label{p.continuity}
    Let $\Phi$ be an OU-GFF in $D\subseteq \C$. Then, almost surely, for any $x\in D$ the function $t\mapsto \gamma^+_x(t)\in [-2,2]$ is continuous. Moreover,  for any $t,s \in \R$
    \begin{align}\label{e.Holder}
    |\gamma_x^+(t)-\gamma_x^+(s)|\leq 2\sqrt{2}\sqrt{|t-s|}.
    \end{align}
\end{prop}
In fact, this result can be improved. As a consequence of our main result regarding the OU-GFF we improve the Holder constant  $2\sqrt{2}$  in \eqref{e.Holder} to $2$, if one considers $\gamma$ instead of $\gamma^+$ (see Remark \ref{r.Holder2}). To properly state this result, we define the $H_0^1(\R)$ energy for a compactly supported smooth function $f:\R\mapsto \R$ as
\begin{align*}
\EE(f):= \int (f'(t))^2dt + \frac{1 }{4 } \int f^2(t) dt = \int \left( f'(t)+\frac{1 }{2 } f(t) \right)^2  dt.
\end{align*}
We define $H_0^1(\R)$ as the completion of $C^1_c(\R)$ with this norm. Let us see that this energy is the one that determines whether there exists a point $x\in D$ such that $\gamma_x=f$.  
\begin{thm}\label{t.PT_of_functions}
    Let $\Phi$ be an OU-GFF in $D\subseteq \C$, then:
    \begin{itemize}
        \item For any $f\in H_0^1(\R)$ such that $\EE(f)<4$, a.s. there exists $x \in D$ such that $\gamma_x(t)=f(t)$ for all $t\in \R$.
        \item Almost surely there is no $x\in D$ such that $\EE(\gamma_x)>4$. In particular a.s. for all $x\in D$ such that $\gamma_x(\cdot)$ is well-defined, $\gamma_x(\cdot)$ belongs to $H_0^1(\R)$
    \end{itemize}
\end{thm}
As a consequence of our proof techniques, we show in Proposition \ref{p.existence_dimension_OUGFF} and Corollary \ref{c.upper_bound_dimension_OUGFF} that the dimension of points $x\in D$ such that $\gamma_x(t)=f$ is $(2-\EE(f)/2)\vee 0$.

\subsubsection{Additive stochastic heat equation field} Let us now discuss our second model: the SHEF. This field is the solution  of
\begin{align*}
\begin{cases}
    \partial_t \Phi= \frac{1 }{2 }\Delta \Phi + \sqrt{2\pi}W, & \text{ in } \R^+\times D\\
    \Phi(0,\cdot) = \Phi_0(\cdot),
\end{cases}
\end{align*}
where $\Phi_0$ is a GFF independent of $W$, a space-time white noise. It can be shown that  the GFF is the stationary measure of this reversible dynamics and thus the solution can also be extended to negative times. This can be seen from the fact that we can express the solution as
\begin{align*}
\Psi(\cdot,t)= \sum_{k\in \N} \beta_k(t) e_k(\cdot) ,
\end{align*}
where $(\beta_k(t))_k$ are independent Ornstein-Uhlenbeck processes in $\R$ with different speeds. That is to say, they are the invariant solutions of
\begin{align*}
d\beta_k(t) = -\frac{\lambda_k}{2} \beta_k(t)dt + \sqrt{\lambda_k}dB^{(k)}(t),
\end{align*}
where $(B^{(k)})_k$ are independent Brownian motions.

In this case, the thickness function of a typical point is discontinuous on every interval (see Remark \ref{r.discontinuous}). However, this irregularity increases the entropy governing the potential spatial locations of thick points. As a consequence the SHEF has super-thick points\footnote{We call these points super thick because a.s. the classical GFF has no points that are more than $2$-thick}: points in space-time whose thickness is greater than $2$.
\begin{prop}\label{p.ST}
    Let $\Psi$ be a SHEF, $\gamma>0$ and define the set of $\gamma$-thick points as
    \begin{align*}
    T_\gamma:=\left \{(x,t): \gamma^+_x(t):=\limsup_{\epsilon \to 0} \frac{ \Psi_\epsilon(t,x)}{-\log(\epsilon)}=\gamma\right \}.
    \end{align*}
    Then, 
    \begin{itemize}
        \item If $\gamma>\sqrt{8}$, a.s. $T_\gamma=\emptyset$, i.e., no $\gamma$-thick points exist.
        \item If $\gamma<\sqrt{8}$, a.s. $T_\gamma\neq \emptyset$ and the Hausdorff dimension of $T_\gamma$ is given by $(3-\gamma^2/4)\wedge (4-\gamma^2/2)$.
    \end{itemize}
\end{prop}
For us, it was surprising that the formula for the Hausdorff dimension loses analyticity at the value $\gamma^2=4$, which marks the transition between thick points and super thick points.

Next, we focus on super-thick points and  introduce a set of $N$-point exceptional times: times where there are $N$ distinct $\gamma$-thick points. For $N\in \N$ and $\gamma \in \R$, we define
\begin{align*}
E_{N,\gamma}= \left \{ t\in \R:\exists(x_i)_{i=1}^N\subseteq D, \text{ such that } \gamma^+_{x_i}(t) = \gamma \text{ and all $x_i\neq x_j$ for $i\neq j$} \right \} .
\end{align*}
The key result for this model is a phase transition for the existence of $N$-point exceptional times as a function of $\gamma$.
\begin{thm}\label{t.PT_SHEF}
    Let $\Phi$ be a SHEF, and $(\gamma^c_N)^2 = \frac{4}{N}+4$, then:
    \begin{itemize}
        \item For any $\gamma< \gamma^c_N$, a.s. the set $E_{N,\gamma}$ is nonempty.
        \item  For any $\gamma> \gamma^c_N$, a.s. the set $E_{N,\gamma}$ is empty.
    \end{itemize}
\end{thm}
Furthermore, in Proposition \ref{NdimT2}, we obtain that the Hausdorff dimension of $E_{N,\gamma}$ for $\gamma < \gamma^c_N$ is given by
\begin{align}\label{e.dim_ENg}
\text{Dim}(E_{N,\gamma})= 1- \left( \frac{\gamma^2 }{4 }-1\right )N. 
\end{align}

\subsubsection{Critical values and Beraha numbers}\label{sss.Beraha}
Let us make a short comment on the values of $\gamma^c_N$, as they suggest a connection to the algebraic structures of the $q$-state Potts model.

By mapping the thickness parameter $\gamma$ to the SLE parameter via $\kappa = \gamma^2$, and subsequently to the 
Fortuin-Kasteleyn (FK) cluster weight through the relation 
$q = 4\cos^2(4\pi/\kappa)$ (Theorem 1.2 of \cite{MSW21}), one recovers the Beraha numbers, 
$B_n = 4\cos^2(\pi/n)$. 

When the parameter $q$ of a Potts model is equal to a Beraha number interesting things happen. In the ferromagnetic regime, they 
identify the specific central charges for which the corresponding Conformal 
Field Theories (CFT) are minimal \cite{martin1991potts}; notably, the first few values 
correspond to canonical objects in probability: $q=0$ (Uniform Spanning Trees), 
$q=1$ (Critical Percolation) and $q=2$ (the FK-Ising model).
For the antiferromagnetic case, 
the Beraha numbers identify the specific points where the infinite continuous 
spectrum of the underlying GFF is truncated, causing the model 
to undergo a second-order phase transition rather than a first-order one 
\cite{Saleur, JacobsenSaleur18}.

The fact that these algebraic degeneracies in 
lattice models align exactly with the existence thresholds for $N$-point 
exceptional times in the SHEF suggests an underlying Coulomb Gas 
structure that remains to be fully elucidated. Even though these possible connections seem profound, in this paper, we make no attempt to rigorously (or non-rigorously) find them.

\subsection{Ideas of the proofs and organization of the paper} To prove continuity of the OU-GFF (Proposition \ref{p.continuity}), we follow the approach of \cite{HMP}. We begin by estimating the modulus of continuity of $\Phi_\epsilon(x,t)$ with respect to $\epsilon$,\, $x$ and $t$. We then provide a tail estimate on the probability of observing large gradients. This method is subsequently refined to prove the second part of Theorem \ref{t.PT_of_functions}.

The existence part of Theorem \ref{t.PT_of_functions} is more abstract. We observe that for any continuous function $\lambda(t)$, the field $\int_{-\infty}^\infty \lambda(t)\Phi(t,\cdot)dt$ has the law of a GFF multiplied by an explicit constant. We then construct the Liouville measure associated with this GFF, which is possible only when this constant is strictly smaller than 2. We show that a point sampled according to this measure follows a path given by an explicit transformation of $\lambda$.

In the case of SHEF, to obtain an upper bound on the dimension of $E_{N,\gamma}$ we need to first compute the dimension on the lifted domain $D^N\times [0,T]$ of the set $(x_1,\dots,x_N,t)$, where $(x_i,t)$ are different super thick points. We then project this set in time to get the upper bound.

The lower bound is less direct: we create a measure such that a typical point chosen according to it is a time with $N$ super thick points. To achieve this, we introduce the $N$-Liouville measure, defined on the lifted domain $D^N\times [0,T]$, as the limit as $\epsilon \to 0$ of
\begin{align*}
e^{\gamma\sum_{k=1}^N \Phi_\epsilon(x_k,t)- N\frac{\gamma^2 }{2 }\log(1/\epsilon)} dt \prod dx_k.
\end{align*}
We then prove that $(x_1,...,x_N,t)$, a typical point of this measure, satisfies that $(x_i,t)$ is a $\gamma$-thick point. The $N$-Liouville measure is then used to obtain the correct lower bound on the dimension of the exceptional times $E_{N,\gamma}$. 

The paper is organised as follows. We start with preliminaries on the GFF and GMC in Section \ref{s.Preliminaries}. Then, in Section \ref{s.OUGFF} we show that the thickness function of the OU-GFF is continuous and in Section \ref{s.OUGFFthickfun} we characterise all possible thickness functions that can be found on an OU-GFF. We finish by studying the properties of the thick points of the SHEF in Section \ref{s.introSHEF}.

\subsection*{Acknowledgments} We are grateful to Jesper Jacobsen for his insights into the relationship between the Beraha numbers and the antiferromagnetic Potts model, and to Tomás Alcalde, who urged us to explore the relations between $\gamma_N^c$ and $q$ that brought the Beraha numbers to light.

 Part of the work was carried out during and supported by the program \emph{Probabilistic methods in quantum field theory} at the Hausdorff Institute of Mathematics funded by the Deutsche Forschungsgemeinschaft (DFG, German Research Foundation) under Germany's Excellence Strategy – EXC-2047/1 – 390685813. The research of F.E.V. was funded by the Deutsche Forschungsgemeinschaft (DFG, German Research Foundation) under Germany's Excellence Strategy EXC 2044/2 –390685587, Mathematics Münster: Dynamics–Geometry–Structure. The research of A.S. was supported by Centro de Modelamiento Matem\'{a}tico Basal Funds FB210005 from ANID-Chile, by Fondecyt Grant 1240884, and by ERC 101043450 Vortex. 


\section{Preliminaries}\label{s.Preliminaries}
Before presenting the preliminaries needed related to the GFF, we introduce some technical results that will be used later in this work. The first tool in this small tool-kit is the following lemma that will be useful to prove one of the main propositions of this work.
\begin{lemma}
\label{l.limsup}
    Let $(a_n)_{n\in\N}, \hs (b_n)_{n\in\N}$ be two real sequences. We have that,
    \begin{align*}
        |\limsup a_n - \limsup b_n| \leq \limsup |a_n-b_n |.
    \end{align*}
\end{lemma}
\begin{proof}
The result follows directly from,
    \begin{align*}
        \limsup |a_n-b_n | \geq \limsup a_n + \liminf - b_n 
        = \limsup a_n - \limsup b_n.
    \end{align*}
\end{proof}

The second lemma we will need is the following one that gives us a characterization of $H^1([-T,T])$. This characterization may already have appeared  somewhere else in the literature, but as we were unable to find a reference, we give an overview of the proof.
\begin{lemma}
    \label{l.integralLimit}
    Consider $f\in C([-T,T])$. For each $N\in\N$, consider $\Delta_N= T2^{-N}$ and \begin{align*}
    I_N=\{t_k = k\Delta_N, \hs k\in \{-2^N,...,2^N\}\}.
    \end{align*} Then $f\in H^1([-T,T])$ if and only if
    \begin{equation}
        \label{integrallimit}
        \lim_{N\goesto \infty}\sum_{k=-2^N}^{2^N-1}
        \Delta_N
        \left(
        \frac{f(t_{k+1})-f(t_k)}{\Delta_N}
        \right)^{2}<\infty.
    \end{equation}
    Furthermore, 
    \begin{align*}
\lim_{N\goesto \infty}\sum_{k=-2^N}^{2^N-1}
        \Delta_N
        \left(
        \frac{f(t_{k+1})-f(t_k)}{\Delta_N}
        \right)^{2} = \int_{-T}^T |f'|^2 dt,
    \end{align*}
    where $f'$ is the weak derivative of $f$.
\end{lemma}

 It is important to note that the limit of \eqref{integrallimit} always exists as the sequence is increasing. This is because, at level $N$ we are computing the minimal energy of any continuous function that takes values $f(x)$ at all $x \in I_N$. Additionally, note that $t_k=t_k^{(N)}$ strongly depends on $n$. We are not making this dependent explicit, unless strictly necessary.
\begin{proof}
The direct implication follows due to the fact that for every $g$ in $H_0^1$ with $g(x) =f(x)$ for all $x\in I_N$ we have that
\begin{align*}
\sum_{k=-2^N}^{2^N-1}
        \Delta_N
        \left(
        \frac{f(t_{k+1})-f(t_k)}{\Delta_N}
        \right)^{2} \leq \int_{-T}^ T (g'(x))^2 dx.
\end{align*}
We conclude by taking $g=f$.

To prove the converse, we start by taking
    \begin{equation*}
        f_N(s) = \displaystyle\sum_{k=-2^N}^{2^N-1}\frac{f(t_{k+1})-f(t_k)}{\Delta_N}\displaystyle\1_{[t_k,t_{k+1}]}(s),
    \end{equation*}
    and showing that this is a Cauchy sequence in $L^2([-T,T])$. Take $N<M\in \N$, and note that as the partitions are strictly decreasing    \begin{equation*}
        \|f_N-f_M\|_{L^2([-T,T])}^2=\left|
        \sum_{k=-2^N}^{2^N-1}
        \Delta_N
        \left(
        \frac{f(t^{(N)}_{k+1})-f(t^{(N)}_k)}{\Delta_N}
        \right)^{2}
        -
    \sum_{j=-2^M}^{2^M-1}
        \Delta_M
        \left(
        \frac{f(t^{(M)}_{j+1})-f(t_j^{(M)})}{\Delta_M}
        \right)^{2}
        \right|.
    \end{equation*}
    This implies that the sequence $(f_N)_{N\in\N}$ is a Cauchy sequence in $L^2([-T,T])$, thus it converges towards a limit function $g$ in both $L^2([-T,T])$ and $L^1([-T,T])$. 
    
    For each $N\in \N$ and $x\in [-T,T]$, taking $k_N$ such that $k_N\Delta_N \leq x<(k_N+1)\Delta_N$ we find that the integral between $-T$ and $x$ is equal to,
    \begin{equation*}
         \left(f(t_{k_N+1})-f(t_{k_N})\right) \frac{x-t_{k_N}}{\Delta_N} + f(t_{k_N})-f(-T).
    \end{equation*}
    From here, since as $N$ tends to infinity, we see that $t_{k_N(x)}$ converges to $x$ and $(x-t_{k_N(x)})/\Delta_N$ is between $0$ and $1$, and since $f$ is continuous by hypothesis, therefore we can conclude that
    \begin{eqnarray}
        \int_{-T}^x g(s)ds = f(x)-f(-T).  
        \label{eq.weakder}
    \end{eqnarray}
    This allows us to conclude that $f'=g\in L^2([-T,T])$ and hence the lemma.
\end{proof}
Notice that assuming that $f\in H^1([-T,T])$, we can conclude that the series given above converges. Hence, the existence of the limit is equivalent to be in $H^1([-T,T])$. Specifically, we have the following lemma.
\begin{lemma}
    \label{l.aproxenergy}
    For each continuous function $f$ with weak derivative in $L^2$, we have that 
    \begin{align*}
    \frac{f(-T)^2 }{2 } +\sum_{k=-2^N}^{2^N-1}
        \Delta_N
        \left(
        \frac{f(t_{k+1})-e^{-\frac{\Delta_N}{2}}f(t_k)}{\Delta_N}
        \right)^{2} &\stackrel{N\to \infty}{\longrightarrow} \frac{f(T)^2 }{2 } + \int_{-T}^T |f'|^2dt  + \frac{1}{4}\int_{-T}^T |f|^2dt\\
        &\geq \int_{-T}^T (f')^2 dx+ \frac{1 }{4 } \int_{-T}^T f^2 dx.
    \end{align*}
\end{lemma}
\begin{proof}
Let us compute
\begin{align*}
\sum_{k=-2^N}^{2^N-1}
        \Delta_N
        \left(
        \frac{f(t_{k+1})-e^{-\frac{\Delta_N}{2}}f(t_k)}{\Delta_N}
        \right)^{2}=\sum_{k=-2^N}^{2^N-1}
        \Delta_N
        \left(
        \frac{f(t_{k+1})-f(t_k)}{\Delta_N} + \frac{1 }{2 }f(t_k) + O(\Delta_N)\right)^2
\end{align*}
   We conclude, by expanding the square, noting that thanks to Lemma \ref{l.integralLimit} both $
        \Delta_N^{-1}(f(t_{k+1})-f(t_k))$ and $f(t_k)$ are converging in $L^2$ to $f'$ and $f$ respectively, and using the fact that 
        \begin{align*}
        2\int_{-T}^T f(t) f'(t) dt = f(T)^2 -f(-T)^2.
        \end{align*}
\end{proof}

\subsection{The GFF}\label{ss.GFF}
We now start with a small review on properties and definitions related to the GFF. We assume the reader has some familiarity with the GFF, therefore, we will skip the proofs unless it is necessary. For a more detailed introduction, we suggest \cite{dimGFF,PowellWerner, BerPow}.

Heuristically, the GFF on a given planar domain $D\subset \R^2$ can be understood as the random Gaussian field $(\Gamma_x)_{x\in D}$ whose covariance is given by $G_D$ the Green's function\footnote{We  normalise the Green's function that makes that $G^D(x,y)\sim -\log(\|x-y\|)$, as  $x\to y$.} related to the Dirichlet Laplace operator $-(2\pi)^{-1}\Delta$ on $D$. Given the blow-up on the diagonal of $G^D$, one could then understand the GFF as a Gaussian process indexed by $C_0^{\infty}(D)$. We specify this in the present definition.
\begin{defn}
    The GFF is the centered Gaussian process $(\langle \Gamma,f\rangle)_{f\in C_0^{\infty}(D)}$ whose covariance structure is as follows
    \begin{equation*}
        \E\left[\langle \Gamma,f \rangle\langle\Gamma,g\rangle\right] = \int_D\int_D G^D (x,y)f(x)f(y)dx dy.
    \end{equation*}
\end{defn}
Notice that in the above definition one could extend the index set to $H_0^{-1}(D)$ intersected with the set of signed measures on $D$ (see Remark 1.27 from \cite{BerPow}).   Furthermore, as used in this definition, $\ipp{\Gamma}{f}$ is only an indexing of a set of random variables, however it should be understood as an inner product. This last statement can be justified as the GFF has a modification living in the Sobolev space $H^{-\epsilon}_0$(Theorem 1.45 of \cite{BerPow}).

An equivalent way of thinking about the GFF, is that of the standard normal random variable in $H_0^1$ with the inner product
\begin{equation}
\label{eq.h01-ip}
\ipn{f}{g}= \frac{1 }{2\pi } \int_{D} \nabla f\cdot \nabla g\, dx= \frac{1 }{2\pi } \int_D f (-\Delta) g\,dx.
\end{equation}

As was already mentioned, there has been an interest in understanding the geometry of the GFF even though it is not well defined point-wise. There are a few ways to overcome this issue; one of them is to proceed by convolving the field with some regularization kernel. To do this, we follow the procedure from Section 3 of \cite{LQGKPZ}. For each given $x\in D$ and $\epsilon >0$ we consider $\mu_{\epsilon,x}$ the uniform measure\footnote{We do this when $B(x,\epsilon)\subseteq D$, if not we take the harmonic measure of $D \cap \partial B(x,\epsilon)$.} on $B(x,\epsilon)$. Rigorously, this is defined as follows.
\begin{defn}
    \label{d.GFFHmes}
    Let $\Gamma$ be a GFF. We denote by $\Gamma_{\epsilon}(x)$ the inner product $\ipp{\Gamma}{\mu_{\epsilon,x}}$, and by $\Gr_{\epsilon,\delta}(x,y)$ the correlation between $\Gamma_{\epsilon}(x)$ and $\Gamma_{\delta}(y)$.
\end{defn}

Note that as $\epsilon\to 0$, the function $\Gr_{\epsilon,\epsilon}(x,x)$ blows-up. The following proposition shows that, uniformly on $D$, the speed of divergence of the above is logarithmic. This result is key as it is used to construct the Liouville measure in \cite{Ber}

\begin{prop}\label{BoundsCorGFF} 
Let $\Gr$ indicate the Green function of $D$. Then, there exist finite constants $k < K$, that uniformly on $x,y\in D$ and $0<\epsilon, \ddd<1$
    \begin{equation*}
    \ln\left (\frac{1}{\epsilon \lor \ddd \lor |x-y|}\right )+k \leq
        \Gr_{\epsilon,\ddd}(x,y) \leq \ln\left (\frac{1}{\epsilon \lor \ddd \lor |x-y|}\right )+K.     
    \end{equation*}
\end{prop}
\begin{proof}
    One can do as in Section 3 from \cite{Ber}, since the planar GFF is a log-correlated field.
\end{proof}

Notice that, a priori, the field $(\Gamma_{\epsilon}(x))_{\epsilon>0,x\in D}$ is only well defined for each given $\epsilon >0$ and $x\in D$, and not as a random function on those parameters. In fact, this field can be seen as a random function with a given Hölder regularity, using Proposition 2.1 from \cite{HMP}.

\begin{prop}[Proposition 2.1 from \cite{HMP}]
\label{p.modcontGFF}
    The process $(\Gamma_{\epsilon}(x))_{\epsilon, x}$ possesses a modification (that we also denote by $\Gamma_{\epsilon}(x)$) such that for every $\alpha \in (0,1/2)$ and $\zeta, \delta>0$ almost surely there exists a constant $C=C(\alpha,\zeta,\delta)$ such that 
    \begin{equation*}
        |\Gamma_{s}(x)-\Gamma_{r}(y)|\leq C\left(
        \ln\left(\frac{1}{r\land s}\right)^{\zeta}
        \right)\frac{\left|(x,s)-(y,r)\right|^{\alpha}}{\left(r\land s\right)^{\alpha+\delta}}.
    \end{equation*}
    Uniformly on $x,y \in D$ and $r,s \in (0,1), 1/2<r/s<2$.
\end{prop}

\subsection{The GMC}

We now give a small review on the works related to the so-called Gaussian multiplicative chaos (GMC). It was initiated by Kahane in \cite{kahane1985chaos} and then independently constructed in \cite{LQGKPZ}. Heuristically, it is the measure whose density with respect to Lebesgue is given by a parameter $\gamma>0$ times the exponential of a GFF. A priori, this quantity is not well defined. To do so, one needs to proceed with an approximation.

\begin{defn}
    \label{d.GMC} Let $\Gamma$ be a GFF, for a given $\epsilon >0$ and $\gamma\in(0,2)$ we define 
    \begin{equation*}
        M^{\gamma}(\Gamma_{\epsilon}, dz) = \exp\left(
        \gamma \Gamma_{\epsilon} - \frac{\gamma^2}{2}\E[\Gamma_ \epsilon ^2]
        \right)dz.
    \end{equation*}
\end{defn}

As it was proven in \cite{Ber}, the above sequence of measures is the good way to approximate the Chaos measure.

\begin{prop}
    \label{p.GMC}
   Take $\gamma \in [0,2)$. The random measure $M^{\g}(\Gamma_{2^{-n}},dx)$ converges almost surely to a non-trivial random measure $M^{\g}(\Gamma,dx)$. Furthermore, for any bounded deterministic function $f$
   \begin{align*}
\int f(x) M^{\g}(\Gamma_{2^{-n}},dx) \to \int f(x)M^{\g}(\Gamma,dx),
   \end{align*}
   where the convergence is in $L^1$. Finally, there exists a sequence of measures $(M^\gamma_{\epsilon_0}(\Gamma))_{\epsilon_0}$ such that $M^\gamma_{\epsilon_0}(\Gamma) \nearrow M^\gamma(\Gamma)$ as $\epsilon_0 \searrow 0$, and that for all $\epsilon_0>0$
   \begin{align*}
\E\left[\iint \frac{1 }{\|x-y\|^{\alpha} } M^\gamma_{\epsilon_0}(\Gamma, dx)M^\gamma_{\epsilon_0}(\Gamma,dy) \right]<\infty ,
   \end{align*}
   for all $\alpha\in (0,2-\gamma^2/2)$
\end{prop}

We suggest \cite{Aru,BerPow}, for a review on the Liouville measure and its relation with the geometry of the GFF. Note that the measure $M_{\epsilon_0}$ is the one that appears from the limit of $I$ in Lemma \ref{l.sec3ber1} and Proposition \ref{p.mainber}, as well as in the measure on good points defined in \cite{Ber}.

\section{ Ornstein-Uhlenbeck GFF and continuity of its thickness function}\label{s.OUGFF}

As stated in the introduction, our goal is to study (natural) dynamics where the GFF is the stationary measure. In this section, we introduce the first dynamic we study:  the Ornstein-Uhlenbeck GFF (OU-GFF). This is the generalization of the Ornstein-Uhlenbeck process to take ``GFF values''. To be more precise, let us properly introduce it.
\begin{defn}[Ornstein-Uhlenbeck GFF]
    \label{defOUGFF}
    Let $(\alpha_k(t))_{k\in\N}$ be an i.i.d. sequence of Ornstein-Uhlenbeck processes with rate $1/2$ and initial condition a standard Gaussian distribution. The Ornstein-Uhlenbeck GFF (from now on OU-GFF) is defined as 
    \begin{equation}
        \label{GFFOU1}
        \vp(t) = \sum_{k\in \N} \alpha_k(t) e_k,
    \end{equation}
    where $(e_k)_{k\in\N}$ is an orthonormal basis for $H_0^1(D)$.
\end{defn}
Recall here that an Ornstein-Uhlenbeck process with rate $1/2$, is the strong solution of the following SDE
\begin{align}\label{eq.OU}
    dX_t= -\frac{1 }{2 }X_t dt + dB_t.
\end{align}

In Definition \ref{defOUGFF} we are fixing a specific basis of $H_0^1(D)$, however the law of an OU-GFF is independent of the chosen basis. This can be proven by simply doing a change of basis and using that the sequence $(\alpha_k)_{k\in\N}$ is an i.i.d. sequence of Ornstein-Uhlenbeck. The other way is what we present here, by characterizing the field $\varPhi$ using its correlations.
\begin{lemma} 
    \label{covougff}
    Take $(e_n)_{n\in \N}$ an orthonormal basis of $H_0^1(D)$ with the inner product of \eqref{eq.h01-ip} to define the OU-GFF $\varphi$. Then, the collection $(\langle \varphi(t), f\rangle_\nabla)_{t\in \R, f\in H_0^1(D)}$ is a centered Gaussian process with covariance given by
    \begin{equation}
        \E[\ipn{\varPhi(t)}{f} \ipn{\varPhi(s)}{g}] 
        = e^{-\frac{1}{2}|s-t|}\ipp{f}{g}_{\nabla}, \ \ \ \text{ for any $s,t \in R$ and $f,g \in H_0^1(D)$.} \label{e.cov_nabla_OU}
    \end{equation}
    As a consequence, for $f,g\in H_0^{-1}(D)$  and $s,t\in \R$
    \begin{equation*}
        \E[\ipp{\vp(t)}{f}\ipp{\vp(s)}{g}] = 
        \frac{1 }{2\pi } e^{-\frac{1}{2}|t-s|}\ipp{f}{(-\Delta)^{-1}g},
    \end{equation*}
     thus the law of $\varPhi$ does not depend on the choice of the basis $(e_n)_{n\in \N}$.
\end{lemma}
\begin{proof}
    We start by noting that \begin{align*}
    \langle \varphi(t), f\rangle_\nabla = \sum_{k\in \N} \alpha_k(t) \langle e_k, f\rangle_{\nabla}.
    \end{align*}
    This immediately implies that $(\langle \varphi(t), f\rangle_\nabla)_{t\in \R, f\in H_0^1(D)}$ is a centered Gaussian process.

    To check \eqref{e.cov_nabla_OU}, we start from \eqref{GFFOU1}. Since for $s,t$  fixed we know that that series converges in $L^2$, we used the independence $(\alpha_k(\cdot))_{k\in \N}$ to see that
    \begin{align*}
        \E[\ipn{\varPhi(t)}{f} \ipn{\varPhi(s)}{g}]
         = \displaystyle\sum_{k\in \N} \E[\alpha_{k}(s)\alpha_k(t)]\ipp{f}{e_k}_{\nabla}
         \ipp{g}{e_k}_{\nabla} = 
        \displaystyle\sum_{k\in \N} e^{-\frac{1}{2}|s-t|} \ipp{f}{e_k}_{\nabla}
        \ipp{g}{e_k}_{\nabla}.
    \end{align*}

To conclude the uniqueness in law, we just note that the law $(\langle \varphi(t), f\rangle)_{t\in \R, f\in H_0^1(D)}$ determines the law of $\varPhi$. 
\end{proof}

Lemma \ref{covougff} implies that one can decompose the OU-GFF at different times in the same way as the Ornstein-Uhlenbeck process.
\begin{lemma}
    \label{claimsep}
    Take $s,t\in \R$ with $s<t$ and $\varPhi$ an OU-GFF, then there is $\Gamma$ a GFF independent of $\phi(s)$ such that
    \begin{equation}
        \label{gaussep}
        \varPhi(t) = e^{-\frac{1}{2}|t-s|}\vp(s) + 
        \sqrt{1 - e^{-|t-s|}}\Gamma.
    \end{equation}
\end{lemma}
\begin{proof}
Take $\varPhi$ an OU-GFF and define
\begin{align*}
\Gamma:=\frac{1 }{\sqrt{1 - e^{-|t-s|}}} \left( \varPhi(t)- e^{-\frac{1}{2}|t-s|}\vp(s)\right ).
\end{align*}
Note that $(\varphi,\Gamma)$ is a Gaussian process. Thus, the Lemma just follows by directly showing that for any $t\in \R$ and  $f, g\in H_0^1$, $\E\left[\langle \vp(t), f\rangle_\nabla \langle \Gamma, g\rangle_\nabla \right]=0 $.

\end{proof}
Now, we define the circle average of the OU-GFF by only averaging out the spatial coordinates.
\begin{defn}
    Let $\vp$ be an OU-GFF, and $\mu_{\epsilon,x}$ as in Section \ref{ss.GFF}. We define the $\epsilon$-circle average seen from $x$ as $\vpe{\epsilon}(x,t):=\ipp{\vp(t)}{\mu_{\epsilon,x}}$.
\end{defn}
We can just follow \cite{HMP} to show the continuity of $\vp_\epsilon$. 
\begin{prop} 
    \label{ContOUGFF}
    The process $(\vpe{\epsilon}(x,t))_{\epsilon>0, x\in D, t\in \R}$ has a continuous modification (which we also identify with $\vp$). Furthermore, 
    for all $\alpha\in (0, 1/2)$ and $\rho, \zeta >0$, a.s. there
    exists a constant $M = M(\alpha, \rho, \zeta)$ such that
    \begin{equation*}
        |\vpe{\epsilon}(x,t) - \vpe{\ddd}(y,s)| \leq M
        \left(\ln\left(\frac{1}{\epsilon\land \ddd}\right)\right)^{\zeta}
        \frac{(|(x,t,\epsilon)-(y,s,\ddd)|)^{\alpha}}{(\epsilon\land\ddd)^{\alpha+\rho}}
    \end{equation*} 
    for all $(x,t), (y,s)\in D\times [0,T]$ and $\epsilon,\ddd \in [0,1]$ such that $1/2 \leq (\epsilon \land \ddd)/(\epsilon\lor\ddd) \leq 2$.
\end{prop}
\begin{proof}
   Take $\epsilon,\delta \in [0,1]$, $x,y \in D$ and $s,t \in \R$ and consider
    \begin{align*}
        \E[|\vpe{\epsilon}(x,t)- \vpe{\ddd}(y,s)|^{2}]&=
        \E[\vpe{\epsilon}^2(x,t) - \vpe{\epsilon}(x,t)\vpe{\ddd}(y,s)] +  
        \E[\vpe{\delta}^2(y,s) - \vpe{\epsilon}(x,t)\vpe{\ddd}(y,s)]\\
        &= \left (\Gr_{\epsilon,\epsilon}(x,x) - e^{-\frac{1}{2}|t-s|} \Gr_{\epsilon,\ddd}(x,y)\right ) + \left(  \Gr_{\delta,\delta}(y,y) - e^{-\frac{1}{2}|t-s|} \Gr_{\epsilon,\ddd}(x,y)\right )\\
        &\leq \frac{C}{\epsilon\land \ddd}(|(x,\epsilon)-(y,\ddd)| + |t-s|) 
    \end{align*}
   where in the last line we used the bounds from Proposition 2.1 \cite{HMP} and Proposition \ref{BoundsCorGFF}. We conclude using the generalized  Kolmogorov's theorem of \cite{HMP} Appendix C.
\end{proof}

 In this paper, the geometric features we will be interested in regarding the OU-GFF are its highest values. They are represented by the thickness function. 
\begin{defn}
    \label{tfd}
    Let $\vp$ be a realization of the OU-GFF.
    For $x\in D$ we define 
    \begin{align}
        \begin{aligned}
        \label{tf} 
        &\gamma_x^{+}(t): = 
        \limsup_{\epsilon\goesto0}
        \frac{\vpe{\epsilon}(x,t)}{\ln(\frac{1}{\epsilon})}, \\
        &\gamma_x^{-}(t):= \liminf_{\epsilon \to 0 } \frac{\vpe{\epsilon}(x,t) }{\ln\left(\frac{1 }{\epsilon }\right )  }, \\
        &\gamma_x(t): = \begin{cases}
         \gamma_x^+(t) & \text{ if } \gamma_x^+(t)=\gamma_x^-(t),\\
        \dagger& \text{else}.\end{cases}
        \end{aligned}
    \end{align}
   We call them the sup-thickness function, inf-thickness function and thickness function respectively.
\end{defn}
One of the main and maybe most surprising results of this paper is that the thickness functions of the OU-GFF are all continuous. 
\begin{prop}
    \label{contTF}
    Almost surely, for every $x\in D$ the function $\gamma^{\sigma}_x(\cdot):\R\goesto\R$ is continuous, for each symbol $\sigma \in\{+,-\}$. As a consequence, a.s., $\gamma_x(\cdot)$ is continuous at every $t \in \R$ where $\gamma_x \neq \dagger$ in an interval around $t$. Furthermore, for any $\sigma \in \{+,-\}$
    \begin{align*}
        |\gamma_x^\sigma (t)- \gamma_x^\sigma(s )| \leq \sqrt{8}\sqrt{|t-s|}.
    \end{align*}
    In particular, for every $x\in D$ and $t\in  \R,\, \g_x(t),\,\g_x^{\pm}(t) \in [-2,2]$.
\end{prop}
In order to prove this result, we first need to approximate the thickness functions in a discrete way. This is essentially the same as Lemma 3.1 of \cite{HMP}.

\begin{lemma}[Modified Lemma 3.1 from \cite{HMP}]
    There exists a deterministic sequence $(r_n)_{n\in\N}$ that goes to $0$ as $n$ goes to infinity such that a.s. for every $(x,t)\in D\times[0,T]$
    \begin{equation*}
        \limsup_{\epsilon\goesto 0}\frac{\vpe{\epsilon}(x,t)}{\ln(\frac{1}{\epsilon})} 
        = 
        \limsup_{n\goesto\infty}\frac{\vpe{r_n}(x,t)}{\ln(\frac{1}{r_n})}.
    \end{equation*}
    Of course, the same holds when using $\liminf$ instead of $\limsup$.
    \label{disclOU}
\end{lemma}

\begin{proof}
    We do as in \cite{HMP} Section 3. Using Proposition \ref{ContOUGFF} we consider the values $\alpha,\zeta,\ddd \in (0,1/2)$, and we take $r_n = n^{-K}$ for some $K\in \N$ such that $0<K<\alpha/\ddd$. Then we can conclude that almost surely there exists some constant $M>0$ such that
    \begin{equation}
        |\vpe{\epsilon}(x,t) - \vpe{r_n}(x,t)|
        \leq
        M\left(\ln\left(\frac{1}{\epsilon \land r_n}\right)\right)^{\zeta},
        \label{poracotar1}
    \end{equation}
    whenever $r_{n+1} \leq \epsilon \leq r_{n}$. We can conclude from here exactly as in the first part of the proof of Lemma 3.1 from \cite{HMP}.
\end{proof}

We now have all the tools to prove the main proposition of this section.

\begin{proof}[Proof of Proposition \ref{contTF}]
   Without loss of generality, we work with $\gamma^+$. Fix $q,\ddd\in \mbb{Q}$ with $\delta>0$, and define the event
    \begin{equation}
        \label{Andeltaq}
        A_n(\ddd,q):=\left\{
        \sup_{x \in D}\sup_{s\in [q,q+\ddd]} |\vpe{r_n}(x,q) - \vpe{r_n}(x,s)|
        \leq \ln\left (\frac{1}{r_n}\right )\rho(\ddd)\right\}.
    \end{equation}
    Where $\rho(\delta)= (\sqrt{8}+\eta)\sqrt{\delta}$ for a given $\eta>0$. We now claim that a.s. $A_n(\delta,q)$ occurs for $n$ large enough.
    \begin{claim}
        \label{compAisadditive}
       For every $\ddd>0$ and $q \in \mbb{\R}$,  we have that a.s. there is an $n_0>0$ such that for all $n\geq n_0$ the event $A_n(\delta,q)$ holds.
    \end{claim}
   Let us first assume the claim and conclude from there. From Claim \ref{compAisadditive}, a.s. there exists some $n_0\in\N$ such that for all $n>n_0$  and any $\delta, q \in \Q$
    \begin{equation}
        \label{Logcont}
        |\vpe{r_n}(x,q) - \vpe{r_n}(x,s)|
        \leq 
        \ln\left(\frac{1}{r_n}\right)\rho(\ddd),  \ \ \  \ \ \hs \forall (x,s)\in D\times[q-\ddd,q+\ddd].
    \end{equation}
    
  From here we can obtain by using Lemma \ref{l.limsup} that a.s.
    \begin{equation*}
        |\gamma^+_x(s)-\gamma^+_x(q)|\leq \rho(\ddd),
        \ \ \  \ \ \hs \forall  (x,s)\in D\times[q-\ddd,q+\ddd].
    \end{equation*}
    Since this is a.s. true for every $\ddd,$ $q$, $\eta\in \mbb{Q}$, we can conclude by density. Therefore, to conclude the proposition we just need to prove the claim.

    \begin{proof}[Proof of Claim \ref{compAisadditive}]
    We start by approximating the domain using squares. Take $\epsilon>0$  and define $\hat{r}_n =r_n^{1+\epsilon}$ for some $\epsilon>0$. And consider the collection of squares $\mcl{Q}_n$ of length-size $\hat r _n$ that intersects $D$.
    For $\square \in \mcl{Q}_n$ let $z_{\square}$ denote its center. 
    From Proposition \ref{ContOUGFF}, we know that for $\alpha, \epsilon,\zeta\in (0,1/2)$ a.s. there exist a constant $C$ such that for all $\square \in \mcl{Q}_n$ any $x\in \square$
    \begin{equation}
        \label{distanceBoxOU}
        |\vpe{r_n}(x,t) - \vpe{r_n}(z_{\square}, t)|\leq 
        C\left(\ln\left (\frac{1}{r_n}\right )\right)^{\zeta}\frac{1}{\hat r_n^{\alpha}}
        |z_\square - x|^{\alpha}\leq 
        C\left(\ln\left(\frac{1}{r_n}\right)\right)^{\zeta}.
    \end{equation}
    From this inequality, and comparing the time evolution of point $x$ and $z_{\square}$ we obtain  that for all $s\in [q, q+\delta]$, 
    \begin{align}
        \nonumber&\sup_{x \in D}\sup_{s\in [q,q+\ddd]} |\vpe{r_n}(x,q) - \vpe{r_n}(x,s)|\\
        & \leq 2C\ln^{\zeta}\left (\frac{1}{r_n}\right )+\max_{\square\in \mcl{Q}_n; \square \cap D \not= \emptyset}
        \sup_{s\in I(q,\ddd)} 
        |\vpe{r_n}(z_{\square},q) - \vpe{r_n}(z_{\square},s)|.
        \label{BoxIneq}
    \end{align}
    And this implies that $\mbb{P}(A^{c}_n(\ddd,q))$ is upper bounded by
    \begin{equation}
        \label{ProboxIneq}
        \mbb{P}\left(
            \max_{\square\in \mcl{Q}_n; \square \cap D \not= \emptyset}
        \sup_{s\in I(q,\ddd)} 
        |\vpe{r_n}(z_{\square},q) - \vpe{r_n}(z_{\square},s)| \geq 
        \ln\left(\frac{1}{r_n}\right)
        \left(\rho(\ddd)-C\ln^{\zeta-1}\left(\frac{1}{r_n}\right)\right)
        \right).
    \end{equation}
  To bound this probability, let us first, just focus on one square, that is to say the event
    \begin{equation}\label{eq.event_gradient_big}
        \left\{\sup_{s\in [q,q+\ddd]} 
        |\vpe{r_n}(z_{\square},q) - \vpe{r_n}(z_{\square},s)|\geq 
        \ln\left(\frac{1}{r_n}\right)
        \left(\rho(\ddd)-C\ln\left(\frac{1}{r_n}\right)^{\zeta-1}\right)\right\}. 
    \end{equation}
    
    To continue, we want to reduce ourselves to a Brownian motion problem. To do this we define
    \begin{align*}
    X_s= \frac{\vpe{r_n}(z_{\square},s) }{\sqrt{\Var(\vpe{r_n}(z_{\square},s))} } \ \ \text{ and } \ \ a_n =\sqrt{\frac{-\ln(r_n) }{\Var(\vpe{r_n}(z_{\square},s)) }
    }.
    \end{align*}
    Since $(X_s)_{s\in \R}$ is a centered Gaussian process with correlations given by $e^{-\frac{1}{2}|t-s|}$ it has the law of an Ornstein-Uhlenbeck process.
    Therefore, using the re-scaling property of the Brownian motion, the event \eqref{eq.event_gradient_big} has the same probability as the event 
    \begin{equation*}
        \label{SupBM1}
        \left\{ \sup_{s\in [0,\ddd]} 
        |e^{-s/2} B_{e^{s}}-B_1| \geq a_n
        \left(\rho(\ddd)-C\ln\left(\frac{1}{r_n}\right)^{\zeta-1}\right)\right \},
    \end{equation*}
    where $B$ is a standard Brownian motion. By the reflection principle of Brownian motion, the probability of the above event is upper bounded by 4 times the probability of the following event
    \begin{equation}
        \label{MB3}
        \left\{ \B_{e^{\ddd}}-B_1 \geq a_n
        \left(\frac{\rho(\ddd)}{2}-C\ln\left(\frac{1}{r_n}\right)^{\zeta-1}\right) \right \}.
    \end{equation}
     Note that$B_{e^{\ddd}}-B_1 \sim \mcl{N}(0,(e^{\ddd}-1))$. Using  that $a_n^2 \leq -\ln(r_n)+o(1)$ and
    $-(a-b)^2 \leq -(a^2+b^2)$ we obtain that 
    \begin{equation*}
        \mbb{P}\left(
        \sup_{s\in I(q,\ddd)} 
        |\vpe{r_n}(z_{\square},q) - \vpe{r_n}(z_{\square},s)| \geq \ln\left(\frac{1}{r_n}\right)\rho(\ddd)
        \right)
        \leq 
        C e^{-\frac{(\sqrt{8}+\eta)^2}{4}\ln(\frac{1}{r_n}) + o (1)}.
    \end{equation*}
     Since  
    $|\mcl{Q}_n|= O(1/r_n^2)$, we use a union bound to conclude that
    \begin{equation*}
        \mbb{P}\left(
            \max_{\square\in \mcl{Q}_n; \square \cap D \not= \emptyset}
        \sup_{s\in I(q,\ddd)} 
        |\vpe{r_n}(z_{\square},q) - \vpe{r_n}(z_{\square},s)| \geq \ln\left(\frac{1}{r_n}\right)\rho(\ddd)
        \right) \leq \hat{C}
        n^{K(2-\frac{(\sqrt{8}+\eta)^2}{2})}.
    \end{equation*}
    Since the above is valid for arbitrarily large $K$, the claim follows from Borel-Cantelli when $K$ is chosen large enough.
\end{proof}
\end{proof}

\section{Possible thickness function for the OU-GFF}\label{s.OUGFFthickfun}

Fix a function $f:\R\mapsto \R$ and an OU-GFF $\Phi$. In this section, we are interested in the following question: is there an $x\in D$ such that $\gamma_x=f$?.

The answer is related to the following energy of a functional
\begin{align}\label{energy}
 \mcl{E}(f) = \begin{cases}
  \int_{\R}|f'+ \frac{1}{2}f|^2ds= \int_\R (f'(s))^2 + \frac{1 }{4 } f^2(s)ds,   &  \text{ if } f\in H_0^1(\R),\\
  \infty & \text{ else.}
 \end{cases}
\end{align}
As Theorem \ref{t.PT_of_functions} states, there is a phase transition for energy $4$.

This section is divided into three parts. In the first one we construct a functional space that will allow us to integrate in time the OU-GFF. The second part is where we show existence of points in the domain that follow a given deterministic path with energy less than $4$. To do this, we use a GMC for a properly chosen GFF, and show (thanks to \cite{Aru}, Section 2) that a typical point chosen according to that measure satisfies the requirements.

We finish this section, showing that there cannot be paths with energy more than $4$ in the OU-GFF. We do this by approximating the energy using Lemma \ref{l.aproxenergy}.

\subsection{Some useful functional spaces}
The objective of this section is to construct an appropriate Hilbert space $\mcl H$ that is isometric to $(H_0^1,\mcl E)$. To do this, we first define in $C_c^\infty (\R)$ the following bilinear form.
\begin{defn}
    \label{d.bilform}
    We define $a: C_c^\infty (\R)\times C_c^\infty (\R) \goesto \R$ as the bilinear function given by 
    \begin{equation*}
        a: (f,g) \mapsto a(f,g) = \int_\R \int_\R e^{-\frac{1}{2}|r-s|} f(s)g(r)drds.
    \end{equation*}
\end{defn}
Let us check that the above bilinear function is in fact an inner product.

\begin{lemma} 
    \label{l.Tisometry}
    For each $\varphi \in  C_c^\infty (\R)$, define $T(\varphi)(t) = (k* \varphi)(t)$, where $k(s)= \exp(-|s|/2)$. Then, we have $a(\varphi, \varphi) = \mcl E (T\varphi)$. Hence, $a$ is actually an inner product and the operator $T$ defines an isometry between $(C_{c}^{\infty}(\R), a)$ and $(H_0^1(\R), \mcl E)$. 
\end{lemma}
\begin{proof}
  Take $f\in C^1(\R)$ such that $\lim_{|t|\to \infty} f(t)=0$, then by integration by parts
    \begin{equation}
    \label{eq.energia_F}
        \mcl{J} (f) := \int_{\R}\left |f'+\frac{1}2 f\right |^2dt= \EE(f).
    \end{equation}
    This implies that if $ \mcl{J} (f) $ then  $f\in H_0^1(\R)$.
    
   Note that $T\varphi \in C^\infty (\R)$. We use this prove the lemma in two steps. We first check that for every $\varphi \in C_c^{\infty}$,  the limit of $T\varphi(t)$ is $0$ when $|t|\goesto \infty$. Then, we check that \eqref{eq.energia_F} is finite. 

    \textit{Step 1}: Notice that by Fatou's Lemma, we have 
    \begin{align*}
        \limsup_{|t|\goesto\infty}|T\varphi(t)| &\leq \limsup_{|t|\goesto\infty}\int_{\R} e^{-\frac{1}2|t-s|}|\varphi(s)|ds \\ &\leq \int_{\R}|\varphi(s)| \limsup_{|t|\goesto \infty}e^{-\frac{1}2|t-s|}ds = 0.
    \end{align*}
    Therefore we have that $T\varphi(\pm\infty) = 0$.
    
    \textit{Step 2}: We compute the derivative of $T\varphi$ to obtain that,
    \begin{align}\label{e.derivative}
        (T\varphi)'(t)  = \int_{\R}\left (\1_{\{(t-s)<0\} } - \frac{1}{2}\right )
        e^{-\frac{1}{2}|t-s|} \varphi(s)ds  = \int_{\R}\1_{\{(t-s)<0\} }
        e^{-\frac{1}{2}|t-s|} \varphi(s)ds - \frac{1}{2}T\varphi.
    \end{align}
    Now we compute the left hand side of \eqref{eq.energia_F} for $T\varphi$ 
    \begin{align*}
        \mcl J (T\varphi) & = \int_{\R}\int_{\R}\int_{\R}\1_{\{(t-s_1)<0\} }
        e^{-\frac{1}{2}|t-s_1|} \varphi(s_1)
        \1_{\{(t-s_2)<0\} }
        e^{-\frac{1}{2}|t-s_2|} \varphi(s_2)ds_2ds_1dt\\&=\int_{\R}\varphi(s_1)\int_{\R}\varphi(s_2)\left( \int_{\R}\1_{\{(t-s_1)<0\} }
        e^{-\frac{1}{2}|t-s_1|} 
        \1_{\{(t-s_2)<0\} }
        e^{-\frac{1}{2}|t-s_2|} dt\right )ds_2ds_1\\
        &= \int_{\R}\int_{\R}\varphi(s_1)\varphi(s_2)e^{-\frac{1}{2}|s_1-s_2|}ds_1ds_2.
    \end{align*}
    We conclude from here. 
\end{proof}

From now on, we also use $\EE$ to refer to \eqref{eq.energia_F}. From the above lemma, the following definition naturally follows.

\begin{defn}
    \label{d.Hspace1}
    The space $\mcl{H}$ is the completion of $C_{0}^{\infty}(\R)$ under the inner product $a$.
\end{defn}

From Lemma \ref{l.Tisometry} we can conclude that the operator $T$ can be extended to a linear isometry between $\mcl H$ and $H_0^1(\R)$. We put this result in the following lemma.
 
\begin{lemma}
    \label{l.Textended}
    The operator $T$ from Lemma \ref{l.Tisometry} can be extended to a linear isometry between $(\mcl H, a)$ and $(H_0^1(\R),\EE)$. In addition, this operator is a bijection.
\end{lemma}
\begin{proof}
    From Lemma \ref{l.Tisometry} we can conclude directly that $T$ can be extended to a linear isometry between $\mcl{H}$ and $H_{0}^{1}(\R)$, which directly implies injectivity. 
    
    To prove that $T$ is surjective and as $(H_0^1(\R), \mathcal E)$ is a Hilbert space, we only need to check that $T(\mcl H)^{\perp} = \{0\}$. Consider $f\in H_{0}^{1}(\R)$ orthogonal to $T(\mcl{H})$. Then in particular for every $\varphi \in C_{0}^{\infty}(\R)$, we have that
    \begin{equation*}
        \int_{\R}\left (f'+\frac{1}{2}f\right)\left (T\varphi'+\frac{1}{2}T\varphi\right )dt = 0.
    \end{equation*}
    Using the value of $T\varphi'+\frac{1}{2}T\varphi$ from the proof of 
    Lemma \ref{l.Tisometry} and Fubini's theorem, we obtain that
    \begin{equation*}
        \int_{\R}\varphi(s)
        \left(\int_{-\infty}^{s}\left (f'+\frac{1}{2}f\right )e^{-\frac{1}{2}|t-s|}dt
        \right)ds = 0.
    \end{equation*}
    Since this equation holds for every $\varphi\in C_{c}^{\infty}(\R)$, we can the conclude that for all $s\in \R$
    \begin{equation*}
        \int_{-\infty}^{s}\left (f'+\frac{1}{2}f\right )e^{-\frac{1}{2}|t-s|}dt = 0 \implies   \int_{s}^{u}\left (f'+\frac{1}{2}f\right )e^{\frac{1}{2}t}
        dt
        = 
        0,
    \end{equation*}
    for every $s<u \in \R$. Thus, we can conclude that almost everywhere
    \begin{equation*}
        f'+\frac{1}{2}f
        = 0.
    \end{equation*}
     We conclude the proof by using that the only weak solution of this equation in $H_0^1(\R)$ is $f=0$.
\end{proof}


\subsection{Existence of points with energy smaller than $4$}

We now use the above functional setting to prove the existence part of Theorem \ref{t.PT_of_functions}, we fix a $\varphi \in \mcl{H}$ and define the field
\begin{equation}
    \label{e.intfield}
    h= h^{\varphi} := \displaystyle\int_{\R} \Phi(s)\varphi(s)ds.
\end{equation}
The above field is actually a GFF multiplied by some constant that depends on $\varphi$ as we can see in the present lemma.
\begin{lemma} 
    \label{l.hfield}
    If $h$ is given as in \eqref{e.intfield} for $\varphi\in C_{c}^{\infty}(\R)$,
    and $\vp$ is the OU-GFF process. Then, we have that
    \begin{equation*}
        h \Os{\mcl{L}}{=} \sqrt{a(\varphi,\varphi)}\Gamma^{\Co},
    \end{equation*}
    where $a$ is the inner product from Definition \ref{d.bilform}.
\end{lemma}
\begin{proof}
  We first note that for any $g\in H_0^1(D)$
   \begin{equation*}
        \ipp{h}{g}_{\nabla} = \displaystyle\int_{\R} \ipp{\Phi(s)}{g}_{\nabla}\varphi(s)ds.
    \end{equation*}
    This together with \eqref{e.cov_nabla_OU} implies that for any $f,g \in H_0^1(D)$
     \begin{align*}
        \E[\ipp{h}{f}_{\nabla}\ipp{h}{g}_{\nabla}] = \ipp{f}{g}_{\nabla} \displaystyle\int_{\R}
        \displaystyle\int_{\R} e^{-\frac{1}{2}|t-s|} \varphi(s)\varphi (t) dsdt=a(\varphi, \varphi)\ipp{f}{g}_{\nabla} .
    \end{align*}
    We conclude then, from the fact that $h$ is a centered Gaussian field.
\end{proof}
From the above lemma we can prove the existence part of Theorem \ref{t.PT_of_functions}. We do this by showing a stronger result.

\begin{prop}\label{p.existence_dimension_OUGFF}
    Let $f\in H_0^1(\R)$ be such that $\mathcal E(f)<4$. Then, almost surely, the dimension of points $x\in D$ such that $\gamma_x=f$ is at least $2-\mathcal E(f)/2$.
\end{prop}

We prove this proposition by introducing an appropiate Liouville measure.
\begin{proof}
    From Lemma \ref{l.Textended} we know that there exists a unique $\varphi \in \mcl H$ such that $f= T\varphi$. Given such $\varphi$, we take $h$ as in \eqref{e.intfield} and denote $\Gamma$ the GFF such that $h = \gamma \Gamma$ given by Lemma \ref{l.hfield} with $\gamma^ 2 = a(\varphi, \varphi)$. From Proposition \ref{p.GMC}, we can now consider the GMC measure $M^{\gamma}(\Gamma_ \epsilon, dx)$ related to $\Gamma$.
    
    Let us now show that a point chosen according to  $M^{\gamma}(\Gamma_ \epsilon, dx)$ is such that $\gamma_x= f$. We do this by following \cite{Aru}, Section 2, and consider the following probability measure  on pairs $(z,\vp)$ of a point in the domain and a random evolution of a generalised function
    \begin{equation*}
        Q_{\epsilon}(dzd\vp) = \frac{1}{\lambda(D)}M^{\gamma}(\Gamma_ \epsilon,dz)\mbb P (d\vp),
    \end{equation*}
    where $\lambda(D)$ indicates the Lebesgue measure of $D$. Let us now discuss the law of $(X,\vp)$ under $Q_\epsilon$:
    \begin{itemize}
    \item The law of $\vp$, given $X$ is
        \begin{align}\label{e.CM_epsilon_OU}
        \vp = \widehat{\vp} + \Gr_{\epsilon,0}(X, \cdot)\int_ \R e^{-\frac{1}{2}|t-s|}\varphi(s)ds,
    \end{align}
    where $\hat{\vp}$ is a standard OU-GFF independent of $X$.
    \item Conditionally on $\vp$, $x$ is chosen proportionally to $M^{\g}(\Gamma_ \epsilon,dx)$.
    \item The marginal law of $\vp$, is absolutely continuous with respect to an $OU$ process with Radon-Nykodim derivative proportional to $M^\gamma(\Gamma_\epsilon,D)$.
    \end{itemize}

    As $M^\gamma(\Gamma_\epsilon,D)$ converges in $L^1$ to $M^\gamma(\Gamma,D)$, we have that  $Q_ \epsilon$ converges weakly to a measure $Q$ where \eqref{e.CM_epsilon_OU} holds with $G_{\epsilon,0}(x,\cdot)$ is replaced by $G^D(x,\cdot)$. Thus, we can conclude that 
    \begin{equation*}
        \gamma_X (t) = \lim_{\epsilon\goesto 0}\frac{\vpe{\epsilon}(X,t)}{\ln(\frac{1}{\epsilon})}
        = \lim_{\epsilon\goesto 0}\frac{\hat{\vp}_{\epsilon}(X,t)}{\ln(\frac{1}{\epsilon})} + 
        \displaystyle\int_{\R} e^{-\frac{1}{2}|t-s|}\varphi(s)ds = f(t).
    \end{equation*}
    Since $\hat{\vp}$ is a standard OU-GFF independent of $X$, the respective limit is 0. And we know that $\gamma_X$ is continuous, therefore, equality holds for all $t\in \R$. As the marginal law of $\vp$ under $\Q$ is absolutely continuous with respect to a OU-GFF we conclude that a point taken according to $\M^\gamma(\Gamma,dx)$ has the given thickness function. 
    
    To obtain the lower bound of the dimension of such points we use the energy method (\cite{PMBM} Chapter 4, or \cite{fracgeom}, also Chapter 4 ) . For that we take $\epsilon_0>0$ close enough to $0$ so that $M^\gamma_{\epsilon_0}(\Gamma)>0$ and from Proposition \ref{p.GMC}, we know for all $\alpha< 2-\mathcal E (F)/2$ a.s.
    \begin{align*}
\iint \frac{1 }{\|x-y\|^\alpha } M^\gamma_{\epsilon_0}(\Gamma,dx)M^\gamma_{\epsilon_0}(\Gamma,dy)<\infty.
    \end{align*}
    \end{proof}

\subsection{Non existence of highly energetic points}
In the last section, we showed a necessary condition for the existence of certain deterministic thickness functions. We would like to know whether the condition is also sufficient. In this section, we show a stronger result. We see that a.s. there are no points such that the energy of any of its thickness functions is bigger than 4. This is done via a counting type argument, following the ideas of Section 3 of \cite{HMP}.

\begin{prop}
    \label{NonExistPath}
Almost surely there are no points $x\in D$ such that $\mcl{E}(\gamma_x)>4$.
\end{prop}

Let us first discuss what this proposition means for the regularity of the thickness functions $\gamma_x$.
\begin{rem}\label{r.Holder2}
    Proposition \ref{NonExistPath} strictly improves the regularity of the thickness function established in Proposition \ref{contTF}, provided we restrict our attention to $\gamma_x$ rather than the upper and lower functions $\gamma^\pm_x$. Furthermore, for any $f\in H_0^1(\R)$, with $\mathcal E(f) <4$
    \begin{align*}
|f(t)-f(s)| \leq \int_s^t |f'(u)| du\leq \sqrt{|t-s|}\sqrt{\mathcal E (f)}\leq 2 \sqrt{|t-s|}.
    \end{align*}
However, do not forget that Proposition \ref{NonExistPath} only studies the thickness function $\gamma$, but not $\gamma^\pm$ the sup or inf-thickness function. We believe Proposition \ref{NonExistPath} should also hold true for those functions, but our proof does not extend to them.
\end{rem}

\begin{proof}[Proof of Proposition \ref{NonExistPath}]
   We first consider $r_n =n^{-K} $ for some $K\in\N$ large enough, and take $\hat r_n=r_n^{1+\epsilon}$ for some fix $\epsilon>0$ and we consider $\mcl Q_n$ as the set of squares whose edges are in  $\hat r_n\Z^2$ and intersect $D$. Then, we take $T\in\N, \hs T\gg1$ and consider the time interval given by $[-T,T]$. In this time domain we consider the following dyadic approximation,
    \begin{equation*}
        I_N = \left\{ T\frac{k}{2^N}: \hs k\in \{-2^N,...,2^N\} \right\}.
    \end{equation*}
    
    Now, we define $t_k = k\Delta_N$, with $\Delta_N = T2^{-N}$ and take a sequence $\{\eta_k: k=-2^{N}-1,...,2^N-1\}$ of increments in $\R^+$ such that $ \sum \eta_k^2>4$. Furthermore, we consider the set of points with $\eta$-energetic paths as
    \begin{align*}
        T^{N,T}_{>\eta}:=
        \left\{
        x\in D: |\gamma_x(t_{-2^N})|\geq  \eta_{-2^N-1} \text{ and } \left|
        \frac{\g_x(t_{k+1})-e^{\frac{-\Delta_N} {2}}\g_x(t_{k})}{\sqrt{\Delta_N}}\right | \geq \eta_k, \hs \forall k=-2^{N},...,2^N-1,
        \right\}\\
        \subseteq \left\{
        x\in D: |\gamma_x(t_{-2^N})|\geq \eta_{-2^N-1}\text{ and }\hs 
       \left | \frac{\g^+_x(t_{k+1})-e^{\frac{-\Delta_N} {2}}\g^-_x(t_{k})}{\sqrt{\Delta_N}} \right |\geq \eta_k,\, \forall k=-2^{N},...,2^N-1, \hs
        \right\}.
    \end{align*}

    By Lemma \ref{claimsep}, we know that
    \begin{equation*}
        \vpn(x,t_{k+1})- e^{-\frac{\Delta_N}2}\vpn(x,t_k) = \sqrt{1-e^{-\Delta_N}}\Gamma^k_{r_n}(x),
    \end{equation*}
    where $\Gamma^k$ is a GFF. Even more, $(\Gamma^k)_{k=-2^N}^{2^N-1}$ together with $\Gamma^{-2^N-1}=\vpn(\cdot, -2^N)$ is an iid sequence of GFFs, this follows directly from computing their correlations. Now, let us define $T^{N,T,n,C}_{>\eta}$ to approximate $\Gamma^{N,T}_{\eta}$ as
    \begin{equation*}
        \left\{
        x\in D: \frac{|\Gamma^{-2^N-1}|}{-\ln(r_n)} \geq \eta_{-2^N-1} - C\ln^{\zeta -1}\left(r_n^{-1}\right)  \text{ and }
        \sqrt{\frac{1-e^{-\Delta_N}}{\Delta_N}}\frac{\left| \Gamma^k_{r_n}(x)\right | }{ \log(r_n^{-1})} \geq \eta_k - C\ln^{\zeta -1}\left(r_n^{-1}\right),  \forall k
        \right\},
    \end{equation*}
    where $\zeta$ comes from Proposition \ref{p.modcontGFF} and $k$ takes values between $-2^{N}$ and $2^N-1$.

    From Proposition \ref{p.modcontGFF}, we have that for every $n_0\in \N$
    \begin{align}
    \label{eq.OUGFF-cover}
        T^{N,T}_{>\eta}\subseteq \bigcup_{C\in \Q^+}\bigcup_{n\geq n_0}\bigcup_{\substack{\square\in \mcl Q_n\\
        z_{\square}\in T^{N,T,n,C}_{>\gamma}}} \square,
    \end{align}
    where $z_{\square}$ is the center of the box $\square$. Let us now check that there is a random $n_0\in \N$, such that for all $n\geq n_0$ there are no $\square \in \mcl Q_n$ with  $z_{\square}\in T^{N,T,n,C}_{>\eta}$. To do this, let us first take $z\in D$ and compute
    \begin{align*}
\P(z\in  T^{N,T,n,C}_{>\eta}) &\leq \prod_{k=-2^N-1}^{k=2^N-1} 
        \mbb P \left(|\Gamma^k_{r_n}(z)| > \left(\eta_k - C\ln^{\zeta -1}\left(r_n^{-1}\right)\right)\ln\left(r_n^{-1}\right)
        \right) \\
        &\leq K \exp\left(-\frac{\log(r_n^{-1})}{2 }\sum_k \eta_k^2 \right ) < Kr_n^{2+\delta},
    \end{align*}
    for some $\delta >0$.

    Let us call $\B_n$ the set of all such boxes whose center is in $T^{N,T,n,C}_{>\eta}$
    \begin{align}
    \label{eq.energetic_boxes_OUGFF}
    \sum_{n\geq m}\E\left[\sharp \B_n \right] = \sum_{n\geq m}\sum_{\square \in \mathcal Q_n} \P(z_{\square} \in  T^{N,T,n,C}_{>\eta})\stackrel{m\to \infty}{\to} 0,
    \end{align}
    for $K$ large enough.  Thus, it follows that  for all $C$ there is a random $n_0$, such that for all $n\geq n_0$, $\mathcal B_n$ is empty.

    We conclude using Lemma \ref{l.aproxenergy} to note that for every function $f$ such that $\mathcal E (f)>4$, there exists $T,N\in \N$ and $(\eta_k)_{k=-2^{N}-1}^{2^{N}-1}$ taking values in $\Q^+$  with $\sum \eta_k^2>4$ such that
    \begin{align*}
    &|f(t_{-2^{N}})|\geq \eta_{-2^N-1}\\
    & |f(t_{k+1})-e^{-\frac{\Delta_N }{2 }}f(t_k)|\geq\eta_{k}\sqrt{\Delta_N}, \text{ for all } k=-2^{N},..., 2^N-1.
    \end{align*}
\end{proof}

Let us note that a careful analysis of the proof implies an upper bound on the dimension of functions with a given thickness.
\begin{cor} \label{c.upper_bound_dimension_OUGFF}
Let $f\in H_0^1$ such that $\mathcal E(f)<4$. Then the Hausdorff dimension of all $x\in D$ such that $\g_x=f$ is $2-\frac{\EE(f)}{2}$.
\end{cor}

\begin{proof}
    Note that the lower bound on the dimension is given in Proposition \ref{p.existence_dimension_OUGFF}. For the upper bound, consider $r_n,\, \hat{r}_n$ and $\mcl Q_n$ as before. Take $f\in H_0^1(\R)$ whose energy $\EE(f) < 4$. Let $I_N$ be the dyadic partition of $[-T,T]$ for some given $T\in \N,\, T\gg 1$. Take $\eta_k = (\Delta_N)^{-1/2}(f(t_{k+1})-e^{-\frac{\Delta_N}2}f(t_k))$ for $k\in \{-2^N,...2^N-1\}$ and $\eta_{-2^N-1}= f(-T)$. Then consider $\hat{\eta}=\sum_k \eta_k^2$.

    We first take the covering given in \eqref{eq.OUGFF-cover}, and then for each $C\in \mbb Q^+$ consider $T^{N,T,C}_{\geq\eta}$ as the $\cup_{n\geq n_o}T^{N,T,n,C}_{\geq\eta}$. Since $T^{N,T}_{\geq\eta}\subseteq \cup_{C\in \mbb Q^+} T^{N,T,C}_{\geq\eta}$, we now proceed to find a cover for $T^{N,T,C}_{\geq\eta}$ for some given $C\in \mbb Q^+$ and conclude using the properties of the Hausdorff dimension.
    
    For $C\in \mbb Q^+$ fixed, we consider the cover of $T^{N,T,C}_{\geq\eta}$ given by the the union of boxes at scales at most $n_0$ for a given $n_0\in\N$ with centers in $T^{N,T,n,C}_{\geq\eta}$. Let us call such cover as $B_{n_0}$. Notice that the cover given by $B_{n_0}$ controls the Hausdorff measure $H^{2-\frac{\hat{\eta}}{2}}_{\delta}(T^{N,T,C}_{\geq\eta})$ for every $\delta\in (\hat{r}_{n},\hat{r}_{n-1})$. And from  we can conclude that
    \begin{align}
        H^{2-\frac{\hat{\eta}}{2}+3\epsilon}_{\delta_{n_0}}\left(T_{>\eta}^{N,T,n,C}\right)
        & \leq
        (\sqrt{2} r_{n-1})^{2-\frac{\hat{\eta}}{2}+3\epsilon}  \#B_{n-1},
    \end{align}
    By \eqref{eq.energetic_boxes_OUGFF} we can conclude that 
    \begin{equation*}
        \sum_{n>m}(\sqrt{2} r_{n})^{2-\frac{\hat{\eta}}{2}+3\epsilon}  \E\left[\#B_{n}\right]
        \leq C \sum_{n>m} r_{n-1}^{2-\frac{\hat{\eta}}{2}+3\epsilon + \frac{\hat \eta^2}{2}- 2(1+\epsilon)}
        \leq C \sum_{n>m} r_{n-1}^{\epsilon} <\infty,
    \end{equation*}
    since $r_n = n^{-K}$ for $K$ large enough the above is finite. This implies
    that as $n$ goes to $\infty$, then $ (\sqrt{2} r_{n-1})^{2-\frac{\hat{\eta}}{2}+3\epsilon}  \#B_{n-1}$ tends to $0$ almost surely. Therefore, the corollary follows.
\end{proof}

\section{Stochastic heat equation}\label{s.introSHEF}
From this section onward, we study the behavior of the thick points of the additive stochastic heat equation field (SHEF). This is the solution of the following stochastic equation
\begin{equation}
    \label{SHE-SPDE}
    \begin{cases}
        \partial_t \Psi = \frac{1}{2}\Delta \Psi +  \sqrt{2\pi}\xi, \\
        \Psi(0) = \Psi_0.
    \end{cases}
\end{equation}
Here $\xi$ is space-time white noise in $D\times \R^+$ and $\Psi_0$ is a GFF independent of $\xi$. 

Let us note that the solution is, in fact, a dynamic whose invariant law is that of the GFF.
\begin{prop}\label{p.sol_SHEF}
    Let $(e_k, \lambda_k)_{k\in\N}$ be the eigenvectors and eigenvalues of $-\Delta$ the Dirichlet renormalised Laplacian so that $(e_k)_k$ is an orthonormal basis of $H_0^1(D)$ when equiped with the inner product \eqref{eq.h01-ip}. Then,
    \begin{align}\label{e.decomposition_SHEF}
        \Psi(\cdot, t) = \sum_{k} \beta_k(t) e_k(\cdot),
    \end{align}
    where $\beta_k(t)$ are independent Ornstein-Uhlenbeck processes with speed $\lambda_k$
    \begin{align}\label{e.OU_SHEF}
        d\beta_k(t) = -\frac{\lambda_k }{2 }\beta_k(t) dt + \sqrt{\lambda_k}dB^{(k)}_t,
    \end{align}
    and initial condition $\beta^0_k$ having normal distribution with variance $1$.
\end{prop}
\begin{proof}
   Recall the notion of solution of a linear SPDE given by Definition 6.1 from \cite{HSPDE}. Since the operator driving equation \eqref{SHE-SPDE} is proportional to minus the Laplacian it is enough to test our equation against $(e_k)_{k\in \N}$ and solve the corresponding system of equations. Note that
    \begin{align*}
        \ipp{\Psi(t)}{e_k}-\ipp{\Psi_0}{e_k}
        = 
        -\frac{1}{2}\int_0^t \ipp{\Psi(r)}{(-\Delta)e_k}dr +  \sqrt{2\pi}\int_0^t \ipp{\xi(dr)}{e_k }
    \end{align*}
    where the integral against the white noise can be interpreted as evaluated with respect to $\1_{[0,t]}(s)e_k(x)$. Since $(e_k)_k$ are orthogonal in $L^2(D)$ with square norm equal to  $2\pi/\lambda_k$, we can see that 
    \begin{equation*}
     \int_0^t \ipp{\xi(dr)}{e_k } \overset{law }{=} \frac{\sqrt{2\pi} }{ \sqrt{\lambda_k}}B^{(k)}(t),
    \end{equation*}
    for $(B^{(k)})_k$ an iid sequence of Brownian motions.
    
     To conclude, we note that we can write $\beta_k := (\lambda_k/2\pi)\ipp{\Psi}{e_k}$ so that \eqref{e.decomposition_SHEF} holds. Using that $\ipp{\Psi}{(-\Delta)e_k}=\lambda_k\ipp{\Psi}{e_k} $, we obtain that $\beta_k$ solves \eqref{e.OU_SHEF} with  i.i.d initial conditions distributed as $ N (0,1)$. The initial conditions follows from the fact that $\ipp{\Psi_0}{e_k} = \ipn{\Psi_0}{(-(2\pi)^{-1}\Delta)^{-1}e_k}$ and that $\Psi_0$ has the law of a GFF.
\end{proof}

\begin{rem}
    If the reader is not comfortable with the SPDE setup, we recommend to see Proposition \ref{p.sol_SHEF} as the definition of the SHEF. Furthermore, this proposition allows us to naturally extend the SHEF to negative times. From here onwards, and when it is useful, we think of the SHEF as an evolution starting from $-\infty$ that is invariant in law by time shifts. In this context, we define the compact approximation of the space-time domain $D_T= D\times [-T,T]$. 
\end{rem}

\subsection{Correlations and continuity}
In this section, we study the space-time correlation of the SHEF and conclude that it has a Hölder modification with respect to the parabolic metric. Some of these results may have already appeared somewhere else, but since we could not find a proper reference, we will prove them here.

\begin{lemma}
\label{coveq}
Fix $\epsilon,\ddd>0 $ and $y\in D$ and define $u(x,t)=u_{\epsilon,\delta}(x,t):=\E[\pe{\epsilon}(x,t)\pe{\ddd}(y,0)]$. We have that $u$ is equal to $\ipp{\mu_{\epsilon,x}}{f(t)}$, where $f$ solves the following heat equation
\begin{equation*}
    \label{detEq}   
   \tag{H}\begin{cases}
        \partial_t h = \frac{1}{2} \Delta h, & \text{ in } D_T \\
    h = \Gr_{\ddd}(\cdot,y), & \text{ in } \partial_p D_T.
    \end{cases}
\end{equation*}
Here, $\partial_p D_T$ indicates the parabolic boundary that is its space-time boundary minus $D\times \{T\}$, and $\Gr_{\ddd}(\cdot,y) = \Gr_{0,\ddd}(\cdot,y)$. Here, one extends $h$ and $\Gr$ to be zero outside $D$.
\end{lemma}
The proof is somehow straight-forward but one has to be careful with the constant. Sometimes it is right to consider the operator as $(e_k)_{k\in \N}$, the eigenfunction of $\Delta$, are such that $\langle e_k, -\Delta e_k \rangle = 2\pi$.

\begin{proof}

It follows from Proposition \ref{p.sol_SHEF} that, for every $s,t \in [0,T]$ and $f,g\in H_{0}^{1}(D)$ 
\begin{align*}
    \E\left[\ipp{\Psi(s)}{f}_{\nabla} \ipp{\Psi(t)}{g}_{\nabla}\right]
    &= \displaystyle\sum_{k\in\N}\ipp{f}{e_k}_{\nabla}\ipp{g}{e_k}_{\nabla}
    \E\left[\beta_k(t)\beta_{k}(s)\right]\\
    &= \sum_{k\in\N}e^{-\frac{\lambda_k}{2}|t-s|}\ipp{f}{e_k}_{\nabla}.
    \ipp{g}{e_k}_{\nabla}.
\end{align*}
As a consequence, when we take $f, g\in L^2(D)$, the respective covariance is
\begin{align*}
    \E\left[ \ipp{\Psi(s)}{f} \ipp{\Psi(t)}{g}\right]&= \displaystyle\sum_{k\in\N}
    e^{-\frac{\lambda_k}{2}|t-s|}\ipp{f}{\hat{e}_k}\ipp{g}{(-(2\pi)^{-1}\Delta)^{-1}\hat{e}_k}\\
       & =
    \ipp{f}{\displaystyle\sum_{k\in\N} 
    e^{-\frac{\lambda_k}{2}|t-s|}\hat{e}_k\ipp{g}{(-(2\pi)^{-1}\Delta)^{-1}\hat{e}_k}},
\end{align*}
where $\hat{e}_k= \sqrt{\frac{\lambda_k }{2\pi }}e_k$ is an orthonormal basis of $L^2(D)$. From here, we see that  
\begin{align}
    \label{eq.Sol_H_Bm}
    u(x,t) = \ipp{\mu_{\epsilon,x}}{\displaystyle\sum_{k\in\N} 
    e^{-\frac{\lambda_k}{2}t}\hat{e}_k\ipp{(-(2\pi)^{-1}\Delta)^{-1}\mu_{\delta,y}}{\hat{e}_k}}.
\end{align}

 On the other hand, let $h$ be the solution of \eqref{detEq}. By solving the equation using separation of variables (see \cite{Ev} Chapter 4) we can write $h$ as 
\begin{equation*}
   h(x,t) = \sum_{k\in\N}e^{-\frac{\lambda_k}{2}t} \hat{e}_k(x)\ipp{\Gr_{0,\delta}(\cdot,y)}{\hat{e}_k}.
\end{equation*}
We finish the proof by noting that the right-hand side of \eqref{eq.Sol_H_Bm} is a duality product between $\mu_{\epsilon,x}$ and $h$
\end{proof}

\begin{rem} It is known that the solution of the heat equation can be represented as an observable of the Brownian motion (see \cite{taylor1996partial}, Proposition 3.3). Therefore, we have that $\E[\pe{\epsilon}(x,t)\pe{\delta}(y,0)]$ is equal to
\label{rem.heatequationBm}
    \begin{equation}
    \label{eq.repcovinspacetime}
        \E[\pe{\epsilon}(x,t)\pe{\ddd}(y,0)] = \E_B\left[\E^{B^x_{\Tt_{\epsilon,x}}}_{\widetilde{B}}\left[\Gr_{\delta}(\widetilde{B}_{t\land\Tt_D},y)\right]\right].
    \end{equation}
\end{rem}

Now that we have a better comprehension of the correlations, we proceed with finding a Hölder modification of the field $\Psi_{\epsilon}$ on the space-time domain. For this we follow the steps of Proposition 2.1 of \cite{HMP}.

\begin{prop}
\label{p.ContSHE}
Take $T>0$ and $\pe{\epsilon}(x,t) = \ipp{\Psi(t)}{\mu_{\epsilon,x}}$ as in Definition \ref{d.GFFHmes}. Then $\Psi_\epsilon$ has a modification (that we denote the same way) such that for every $0<\alpha<0.5$, and $\ddd,\zeta>0$, almost surely, for $r<\epsilon$, there exists $C=C(\alpha,\ddd, \zeta)$ such that for all $x,y \in D$, $s,t \in [-T,T]$ and $\epsilon,r>0$ with $1/2<r/\epsilon<2$
\begin{equation*}
    |\Psi_{r}(x,t)-\Psi_{\epsilon}(y,s)| \leq
    M
    \left(\ln\left (\frac{1}{r}\right )\right)^{\zeta} 
    \frac{1}{r^{\alpha+\ddd}}
    \left(|\epsilon-r| + d_p((x,t),(y,s))\right)^{\alpha}.
\end{equation*}
 Here, $d_p((x,t),(y,s)) = \sqrt{|t-s|}\lor |x-y|$ for $(x,t),(y,s)\in D_T$ indicates the parabolic metric of $D_T$.
\end{prop}

\begin{proof}[Proof]
    This result is an application of the modified Kolmogorov extension theorem given in \cite{HMP} Appendix C.
    To apply this proposition we need to estimate $\E\left[|\Psi_{\epsilon}(x,t) - \Psi_{\delta}(y,s)|^2\right]$ for $x,y \in D$. Without loss of generality we know we can focus on $| \E[\Psi_{\epsilon}(x,t)^2] - 
    \E[\Psi_{\epsilon}(x,t)\Psi_{\ddd}(y,s)]|$ and assume that $t>s$. By Remark \ref{rem.heatequationBm} we have that 
    \begin{align*}
        \left| \E[\Psi_{\epsilon}(x,t)^2] - 
    \E[\Psi_{\epsilon}(x,t)\Psi_{\ddd}(y,s)]\right|
    &=
    \left| \Gr_{\epsilon,\epsilon}(x,x) - \E_B\left[\E^{B^x_{\Tt_{\epsilon,x}}}_{\widetilde{B}}\left[\Gr_{\delta}(\widetilde{B}_{r\land\Tt_D},y)\right]\right]\right| \\ 
    &\leq 
    \E_{B}\left[\E^{B^x_{\Tt_{\epsilon,x}}}_{\widetilde{B}}\left[\left| G_{\epsilon}(B_{\Tt_{\epsilon,x}},x) - \Gr_{\delta}(\widetilde{B}_{r\land\Tt_D},y)\right|\right]
    \right],
    \end{align*}
    where $r= |t-s|$. To control this, we make the following claim
    \begin{claim}
        \label{c.prioronGreen}
        There exists a positive constant $C$ such that uniformly on $0<\epsilon, \delta <1$ and $x,y, z,z'\in D$, we have that 
        \begin{equation*}
            \label{eq.prioronGreen}
             |\Gr_{\epsilon}(z,x)-\Gr_{\delta}(z',y)| \leq C\frac{1}{\epsilon\land \delta} \left(|z-z'|+|x-y|+|\epsilon-\delta\right).
        \end{equation*}
    \end{claim}
    By assuming this claim, and that $\Gr_{\delta}(x,y)=0$ when $x\in \partial D$ we have that, the above is equal to
    \begin{align*}
        \E_{B}\left[\E^{B^x_{\Tt_{\epsilon,x}}}_{\widetilde{B}}\left[\left| G_{\epsilon}(B_{\Tt_{\epsilon,x}},x) - \Gr_{\delta}(\widetilde{B}_{r\land \Tt_D},y)\right|\right]
        \right] 
        &\leq 
        \frac{C}{\epsilon\land{\delta}}\left(|\epsilon-\delta|+|x-y|+
        \E_{B}\left[\E^{B^x_{\Tt_{\epsilon,x}}}_{\widetilde{B}}\left[
        \left|B^x_{\Tt_{\epsilon,x}}-\widetilde{B}_{r\land\Tt_D}\right|\right]\right]\right) \\
        &\leq 
        \frac{C}{\epsilon\land{\delta}}\left(|\epsilon-\delta|+|x-y|+
        \E_{B}\left[\E^{B^x_{\Tt_{\epsilon,x}}}_{\widetilde{B}}\left[
        \left|B^x_{\Tt_{\epsilon,x}}-\widetilde{B}_{r}\right|\right]\right]\right)\\ 
        &\leq 
        \frac{C}{\epsilon\land{\delta}}\left(|\epsilon-\delta|+|x-y|+
        \sqrt{r}\right).
    \end{align*}
    Here, from the first line to the second line we have used that $(|x-\widetilde{B}_{r}|)_{r>0}$ is a sub-martingale. The last bound is from $\E[|B_r|]\leq \sqrt{\E[|B_r|]^2}= \sqrt{r}$. By Lemma C1 from \cite{HMP} we can conclude. To finish the proof, it remains only to prove the claim. We do so now.
    
    \begin{proof}[Proof of Claim \ref{c.prioronGreen}]
    From Proposition 3.1 of \cite{LQGKPZ} we know that 
    \begin{equation*}
        G_{0,\epsilon}(z,x) = -\ln\left(|z-x|\lor \epsilon\right) - \widetilde{g}_{\epsilon}(z,x),
    \end{equation*}
    where $g_{\epsilon}(z,\cdot )$ is the harmonic extension in $D$ of the restriction $-\ln\left(|z-\cdot|\lor \epsilon\right)$ to $\partial D$. We focus on the logarithmic term since we can argue by the maximum principle to conclude the desired (see\cite{Ev} Chapter 1).
    \begin{align}
        |G_{0,\epsilon}(z,x)-G_{0,\delta}(z',y)| \leq |\ln\left(|z-x|\lor \epsilon\right)-\ln\left(|z'-y|\lor \delta\right)| \leq \frac{|z-z'|+|x-y|}{\epsilon\land \delta}
    \end{align}
    where for the last inequality we have applied the mean value theorem. Therefore, the claim follows.
    \end{proof}
    
\end{proof}

As in the case of the OU-GFF, we can, now, discuss the evolution of the thick points of the SHEF. To do this we define its $\gamma$-thick points as follows.

\begin{defn}
   Take $\gamma>0$  and $\Psi$ the SHEF. We define 
    \begin{equation*}
        T^{\gamma} := \left\{ (x,t)\in D_\infty:\hspace{0.1cm} 
        \limsup_{\epsilon\goesto 0} 
        \frac{\Psi_{\epsilon}(x,t)}{\ln(\frac{1}{\epsilon})} = \g \right\},
    \end{equation*}
    the set of $\g$-thick points in $D_\infty$. Recall that $\pe{\epsilon}(x,t)$ is the circle average of $\Psi(\cdot,t)$ evaluated at $x\in D$.
    \label{tpSTD}
\end{defn}

\subsection{Gaussian multiplicative chaos on the SHEF}
\label{GMCfortheSHEfield}
The main goal of this section is to construct the GMC measure for $\Psi$ the SHEF. It is well known from \cite{Ber}, that this measure can be constructed
for log-correlated fields. However, $\Psi$ is not a log-correlated field in the usual sense. Therefore, we will have to do some extra work to confirm the existence of this measure. This means that in this section our goal is to prove the following proposition.
\begin{prop}[GMC measure]
    Take $\gamma^2<8$, then the measure 
    \begin{equation}
    \label{d.shef1}
        \mu^{\gamma}_{\epsilon}(dz) = \exp\left(\g\pe{\epsilon}(z)- 
        \frac{\g^2}{2}\E[\pe{\epsilon}^2(z)]\right)dz
    \end{equation}
    converges in probability as $\epsilon \to 0$ to a random finite positive measure $M^{\g}:=M^{\g}(dxdt;\Psi)$ in $D_T$.
\label{p.gmcahe}
\end{prop}
Since the demonstration of this is an extension of the work done by Berestycki in \cite{Ber} using the same techniques, we set the proofs in Appendix \ref{Appen.A}. In particular, we will use Theorem \ref{t.generalgmc}, therefore, we need to check the following hypotheses, 
\begin{itemize}
    \item[(H1)]\label{h.I} The correlations of $(h_{\epsilon}(x))_{\epsilon\in(0,1), x\in D}$ fulfil
    \begin{equation*}
        \Cov(h_{\epsilon}(x),h_{\epsilon}(y)) = -\ln\left(d(x,y)\lor \epsilon\lor \delta \right) + O(1),
    \end{equation*}
    uniformly for $x,y\in D$ and $\epsilon, \delta \in [0,1]$.
    \item[(H2)]\label{h.II} For a given Radon measure $\sigma$ on $\overline{D}$, there is a $\rho >0$ such that
    \begin{equation*}
        \int\int_{\overline{D\times D}}\frac{\sigma(dx)\sigma(dy)}{d(x,y)^{\rho- \epsilon}} < \infty,
    \end{equation*}
    for every $\epsilon>0$.
\end{itemize}

In the following proposition we check \Hs{h.I}{(H1)} for the parabolic metric. This result is similar to the one appearing in \cite{CG1}, the main difference is the approximation: we use the circle average in space, and they use a bump function generated by the heat kernel.
\begin{prop}\label{CG1P3.1}
    Consider $u_{\epsilon, y}(x,t)=\E[\pe{\epsilon}(y)\pe{\epsilon}(x,t)]$. Then, there are $c<C\in \R$, such that
    \begin{equation*}
        c+ \ln\left (\frac{1}{|x-y|\lor \sqrt{|t|}\lor \epsilon}\right ) 
        \leq 
        u_{\epsilon, y}(x,t)
        \leq 
        C+ \ln\left(\frac{1}{|x-y|\lor \sqrt{|t|}\lor \epsilon}\right).
    \end{equation*}
    The upper bound is uniformly on $\epsilon,t,x,y $, meanwhile the lower bound depends on the distance of the boundary of $x$ and $y$.
\end{prop}
\begin{proof}
    We first prove the lower bound by using \eqref{eq.repcovinspacetime} and that $\Gr(x,y) = -\ln(|x-y|) - g(x,y)$, where $g\in C(D\times D)$. We first only focus on the logarithmic therm on \eqref{eq.repcovinspacetime}, by Jensen's inequality it follows that
    \begin{align*}
        \E_{B}^x\left[\E_{\widetilde{B}}^{B_{\Tt_{\epsilon,x}}}\left[\E^{y}_{W}\left[
        -\ln\left(|\widetilde{B}_{t\land \Tt_{D}}-W_{\Tt_{\epsilon,y}}|\right)
        \right]\right]\right]
        &\geq 
        -\ln\left({\E_{B}^x\left[\E_{\widetilde{B}}^{B_{\Tt_{\epsilon,x}}}\left[\E^{y}_{W}\left[|\widetilde{B}_{t\land \Tt_{D}}-W_{\Tt_{\epsilon,y}}|
        \right]\right]\right]}\right).
    \end{align*}
    From here we now proceed to find an appropriate  upper bound of the expectation on the right-hand side to conclude. To do so, let us write $\widetilde B = \widehat B + x + B_{\tau_\epsilon,0}$ and $W= \widehat W + y$, where $\widehat B$, \, $\widehat W$ and $B$ are independent Brownian motions started from 0. Then, use that $(u+v)^2\leq 2 (u^2 + v^2)$ to see that
    \begin{align*}
        \E_{B}^x\left[\E_{\widetilde{B}}^{B_{\Tt_{\epsilon,x}}}\left[\E^{y}_{W}\left[
        |\widetilde{B}_{t\land \Tt_{D}}-W_{\Tt_{\epsilon,y}}|
        \right]\right]\right]
        &\leq 2\sqrt{\E\left[
        |\widehat{B}_{t\wedge \tau_{D-x+y-B_{\tau_\epsilon,0}}}|^2+|B_{\Tt_{\epsilon,0}}-\widehat W_{\Tt_{\epsilon,0}}|^2
        \right] + |x-y|^2} \\
        &\leq 2\sqrt{3}\left(\sqrt{\E\left[|\widehat{B}_t|^2\right]} \lor \sqrt{\E\left[|B_{\Tt_{\epsilon,0}}-\widehat W_{\Tt_{\epsilon,0}}|^2\right]} \lor |x-y|
        \right)
        \\
        &\leq
        2\sqrt{6} \left(\sqrt{t} \lor |x-y| \lor \epsilon \right),
    \end{align*}
    where in the second inequality we have used that $|\widehat B_t|^2$ is a sub-martingale.

    Coming back to the lower bound on $u_{\epsilon,y}(x,t)$, we need now to treat the term coming from $g(x,y)$. Note that as $g(\cdot, y)$ is the harmonic extension inside $D$ of the restriction of $-\log(|\cdot-y|)$ to $\partial D$ (and extended to $\R^2\sim D$ in the same way), we have that
    \begin{equation*}
    \E_{B}^x\left[\E_{\widetilde{B}}^{B_{\Tt_{\epsilon,x}}}\left[\E^{y}_{W}\left[
        g\left(\widetilde{B}_{t\land \Tt_{D}},W_{\Tt_{\epsilon,y}}\right)
        \right]\right]\right]\leq
        \E_{B}^x\left[\E_{\widetilde{B}}^{B_{\Tt_{\epsilon,x}}}\left[\E^{y}_{W}\left[
        g\left(\widetilde{B}_{t},W_{\Tt_{\epsilon,y}}\right)
        \right]\1_{t<\tau_D}\right]\right]  -\log(d(y,\partial D))
    \end{equation*}
    
    Recalling now that $g(\cdot,\cdot)$ is symmetric, by fixing $\widehat B_t$  we can use the mean-value property (see \cite{Ev} Chapter 1), to see that
    \begin{equation*}
    \E_{B}^x\left[\E_{\widetilde{B}}^{B_{\Tt_{\epsilon,x}}}\left[\E^{y}_{W}\left[
        g\left(\widetilde{B}_{t},W_{\Tt_{\epsilon,y}}\right)\right]\1_{t<\tau_D}\right]\right]= 
        \E_{B}^x\left[\E_{\widetilde{B}}^{B_{\Tt_{\epsilon,x}}}\left[g(\widetilde{B}_{t},y)\1_{t<\tau_D} \right]\right]\leq -\log(d(y,\partial D)),
    \end{equation*}
     Therefore, the lower bound is uniform as long as we fix a compact set $A\subseteq D$ such that $y\in A$.

    To prove the upper bound, from \eqref{eq.repcovinspacetime} we first focus on
    \begin{equation}
    \label{eq.0estimatePE}
        f(x,t) = \E^x\left[\Gr_{\epsilon}\left(B^x_{t\land \Tt_D},y\right)\right],
    \end{equation}
   and we state the following claim.
    \begin{claim}
        \label{c.prioronconvgreen}
        There exists a positive finite constant $K$ such that uniformly on $\epsilon>0$ and $x,z\in D$ we have that
        \begin{equation*}
            \Gr_{\epsilon}(z,y) = \E\left[\Gr(z,B_{\Tt_{\epsilon,y}^D})\right] \leq \ln\left(\frac{1}{|z-y|\lor \epsilon}\right) + K.
        \end{equation*}
    \end{claim}
    Let us first prove the upper bound using this claim. A direct consequence of it is that $f(x,t)$ is uniformly upper bounded by
    \begin{align*}
        K + \E\left[\ln\left(\frac{1}{|B^x_{t}-y|\lor \epsilon}\right)\right] = 
        K + \ln\left(\frac{1}{\sqrt{t}}\right) + \E\left[\ln\left(\frac{1} {\left|\frac{(x-y)}{\sqrt{t}}-B_1\right |\lor \frac{\epsilon}{\sqrt{t}}}\right)\right] . 
    \end{align*}
    Here we split the integral using the set
    \begin{align*}
    A(x,y,t)=\left \{\left |\frac{x-y }{\sqrt{t} }-B_1\right | <1\lor \frac{|x-y| }{\sqrt{t}} \right \}
    \end{align*} and its complement. Since $\ln(|x|)\in L^1(\nu(dx))$ for $\nu$ the standard normal distribution in $\R^2$, we have that
    \begin{equation}
    \label{eq.1estimatePE}
        \E\left[\ln\left(\frac{1}{|\frac{(x-y)}{\sqrt{t}}-B_1|\lor \frac{\epsilon}{\sqrt{t}}}\right)\1_{A(x,y,t)}\right]  \leq \widetilde{K},
    \end{equation}
    where the constant is uniformly on $x,y\in \R^2$ and $t,\,\epsilon \geq0$. The other term is direct to upper bound as follows
    \begin{equation}
    \label{eq.2estimatePE}
      \ln\left(\frac{1}{\sqrt{t}}\right)  +\E\left[\ln\left(\frac{1}{|\frac{(x-y)}{\sqrt{t}}-B_1|\lor \frac{\epsilon}{\sqrt{t}}}\right)\1_{A(x,y,t)^{c}}\right]  \leq \ln\left(\frac{1}{|x-y|\lor \sqrt{t}\lor \epsilon }\right).
    \end{equation}
    Using \eqref{eq.1estimatePE} and \eqref{eq.2estimatePE} in \eqref{eq.0estimatePE} we can conclude that
    \begin{align*}
        f(x,t) \leq C + \ln\left(\frac{1}{|x-y|\lor \sqrt{t}\lor \epsilon}\right),
    \end{align*}
    with $C= K + \widetilde{K}$. Therefore, by Lemma \ref{coveq}
    \begin{align*}
        \E[\pe{\epsilon}(x,t)\pe{\delta}(y,0)] 
        &=
        \ipp{\mu_{\epsilon,x}}{f(\cdot,t)} 
        \leq
        C + \E\left[\ln\left(\frac{1}{|B^x_{\Tt_{\epsilon,x}}-y|\lor \sqrt{t}\lor \epsilon}\right)\right],
    \end{align*}
    from here we can conclude arguing similarly to the proof of Claim \ref{c.prioronconvgreen}. To avoid repetition, we now prove the claim and conclude the proof.
    \begin{proof}
    [Proof of Claim \ref{c.prioronconvgreen}]
        For simplicity we suppose that $D\subseteq B(0,1/2)$, as the constant $K$ can be taken as zero from this supposition. since for every $z\in D$ the function $\Gr(z,\cdot)$ is positive and $0$ in the boundary, we have that
        \begin{align*}
            \Gr_{\epsilon}(z,y) = \E\left[\Gr(z,B^y_{\Tt_{\epsilon,y}})\right]\leq \E\left[ G^{B(z,1)}(z,B^y_{\Tt_{\epsilon,y}}) \right] \leq \log\left(\frac{ 1}{|z-y|\vee \epsilon } \right)  .
        \end{align*}
        Here $\Tt_{\epsilon,x}$ indicates the exit time of $B(x,\epsilon)$ and we use that $G^{B(z,1)}(z,w)= -\log(|z-w|)$.
    \end{proof}
\end{proof}
\begin{rem}
    Notice that the above also holds when considered $u(x,t)= \E\left[\pe{\epsilon}(x,t)\pe{\delta}(y,s)\right]$, with $\epsilon\lor \delta$ in place of $\epsilon$ and $|t-s|$ instead of $t$. The proof is the same and we only consider the case $\epsilon=\delta$ and $s=0$ for simplicity.
\end{rem}

The above proposition tells us that the SHEF $\Psi$ convolved with the circle average measure satisfies \Hs{h.I}{(H1)}. Therefore, to prove Proposition \ref{p.gmcahe}, it remains to verify \Hs{h.II}{(H2)}. We do this in the following proposition.

\begin{prop}
    \label{p.hII} 
    For every $\epsilon >0$ we have that
    \begin{equation*}
        \int_{D_T}\int_{D_T} \frac{dxdtdydu}{\left(|x-y|\lor \sqrt{|t-s|}\right)^{4-\epsilon}} <\infty.
    \end{equation*}
\end{prop}

\begin{proof}
Fix $(y,s) \in D_T$ and let $R>0$ be the diameter of $D_T$. We only need to focus on
\begin{align}\label{e.integral_SHEF_GMC}
\int_{D_T} \frac{dx dt }{ \left (|x-y|\vee \sqrt{|t-s|}\right )^p} \leq \int_{B((y,s), R)} \frac{dx dt }{ \left (|x-y|\vee \sqrt{|t-s|}\right )^p} \leq \int_0^R\int_0^R \frac{r dr du }{(r\vee \sqrt{u})^p }.
\end{align}

Recall that we are only interested in $p<4$ and note that
    \begin{equation*}
        \int_0^R\int_0^R \frac{rdrdu}{(r\lor \sqrt{u})^p} 
        =\int_0^R\int_{r^2}^R \frac{1}{u^{\frac{p}{2}}}durdr +
    \int_0^R\int_{\sqrt{u}}^R \frac{1}{r^p}rdrdu.
    \end{equation*}
By computing the integral above we obtain that  \eqref{e.integral_SHEF_GMC} is upper bounded by\begin{equation*}
        \frac{2}{p-2}\int_0^R  \frac{1}{r^{p-3}} - \frac{r}{R^{\frac{p-2}{2}}}dr + \frac{1}{p-2}\int_0^R \frac{1}{u^{\frac{p-2}{2}}} -  R^{2-p}  du<\infty.
    \end{equation*}
    As the upper bound does not depend on $(y,s)$, we conclude.
\end{proof}

Proposition \ref{p.gmcahe} swiftly follows.
\begin{proof}[Proof of Proposition \ref{p.gmcahe}]
It is a consequence of Theorem \ref{t.generalgmc} as we have already checked \hyperref[h.I]{(H1)} and \hyperref[h.II]{(H2)} for $\rho=4$ in Proposition \ref{CG1P3.1} and \ref{p.hII} respectively.
\end{proof}

\subsection{The SHEF seen from a GMC typical point}
\label{GMCYTP} We now see that a point chosen according to the GMC measure is a.s. a $\gamma$-thick point of the SHEF. This is stated in the following proposition.
\begin{prop}
    \label{MainGMCTP}
    Let $\Psi$ be a SHEF and $\gamma\in(0,2\sqrt{2})$. Then a.s. a point sampled from $M^{\g}(\Psi,dxdt)$ is a $\g$-thick point.
\end{prop}

To prove this result we follow the ideas of Section 2 of \cite{Aru}, more specifically, in particular we follow the argument presented in Section 3.3 of \cite{LQGKPZ}. For $\epsilon>0$, we define the $\epsilon$-rooted measure as the probability measure in 
$Ps(D,[0,T])\times D_T$, where $Ps(D,[0,T])$  is the space of measurable functions from $[0,T]$ to 
the space of Schwartz distributions $\mcl{D}'(D)$ characterised by
\begin{equation}
    \mbb Q_{\epsilon}^{\gamma}(d\Psi,dxdt)= \exp\left(\gamma\pe{\epsilon}(x,t)-\frac{\gamma^2}{2}\E[\pe{\epsilon}^2(x,t)]\right)
     \frac{dxdt}
     {\text{Leb}(D_T)}\mbb{P}(d\Psi).
     \label{rootedmeasure}
\end{equation}
In fact, we are interested in the limit of this measure. This limit is called the rooted measure, also 
known as a ``Peyriére measure'', as it first appeared in \cite{PT}.

Since we want to study the limit of \eqref{rootedmeasure}, we first need
to characterizes the $\epsilon$-rooted measure.
\begin{lemma}
    \label{l.EL}
    Take $(\Psi,x,t)$ sampled from $\mcl{Q}_{\epsilon}^{\g}$. The conditional law of $(x,t)$ given $\Psi$  is 
    \begin{equation}
    \label{GMCcon1}
        \nu_{\epsilon}^{\g,N}(dxdt):=\frac{\exp\left(\g \pe{\epsilon}(x,t)
        -\frac{\g^2}{2}\E[(\pe{\epsilon}(x,t))^2]\right)
        }{\mu_{\epsilon}^{\g}(D_T)}dxdt.
    \end{equation}
    And on the other hand, the law of $\Psi$ conditionally on $(x,t)$, is equal to 
    \begin{equation}
        \label{shifetdM1}
        \hat{\Psi} + \g \Cov(\cdot, \pe{\epsilon}(x,t)),
    \end{equation}
    where $\hat{\Psi}$ has the law of the SHEF on $Ps(D, [0,T])$.
\end{lemma}
\begin{proof}[Proof]
    We use Theorem 8.5 of \cite{Kal} to see that
    \begin{align*}
\mbb Q_\epsilon ^\gamma(dxdt \mid \Psi) \propto \exp\left(\g \pe{\epsilon}(x,t)
        -\frac{\g^2}{2}\E[(\pe{\epsilon}(x,t))^2]\right) dx dt.
    \end{align*}
    From here it directly follows \eqref{GMCcon1}.

    For the other conditional probability, we use the same trick
    \begin{align*}
\mbb Q_\epsilon ^\gamma( \Psi \mid x,t) \propto \exp\left(\g \pe{\epsilon}(x,t)
        -\frac{\g^2}{2}\E[(\pe{\epsilon}(x,t))^2]\right) d\P.
    \end{align*}
    Noting that the proportionality constant is 1. We obtain \eqref{shifetdM1} from Cameron-Martin's theorem as $\P$ is a Gaussian measure.
    \end{proof}
From this first result, we can actually conclude the convergence. 
\begin{prop}
    The measure $\mbb Q_{\epsilon}^{\g}$ converges weakly to a measure $\mbb Q^{\g}$ when $\epsilon\goesto 0$. Even more, it holds that for every continuous and bounded function $F$  from $Ps(D,[0,T])\times D_T$,
    \begin{equation*}
        \E_{\mbb Q^{\gamma}}\left[F(\Psi, (x,t))\right]
        = 
        \E_{\mbb P\times \mbb U} \left[F( \Psi+\gamma \Cov((x,t),\cdot),(x,t))\right],
    \end{equation*}
    where $\Cov\left((x,t),(y,s)\right)$ is the limit as $\epsilon \to 0$ of $\Cov\left(\Psi_{\epsilon}(x,t),\Psi_{\epsilon}(y,s)\right)$, and $(x,t)$ are chosen according to $\mbb U$ the law of a uniform random variable in $D_T$. Furthermore, the conditional law under $\mbb Q^{\gamma}$ of $(x,t) $ given $\Psi$ is that of a random variable taken from the normalized $M^\gamma(\Psi, dxdt)$.
    \label{p.SHEconv}
\end{prop}
\begin{proof}
    We already know that for all $\epsilon>0$ we have that
    \begin{equation*}
        \E_{\mbb Q_{\epsilon}^{\g}}\left[F(\Psi, (x,t))\right] = 
        \E_{\mbb{P}}
        \left[\int_{D_T} F(\Psi+ \g \Cov(\cdot, \pe{\epsilon}(x,t)), (x,t)))\frac{dxdt}{\text{Leb}(D_T)}\right].
    \end{equation*}
      Using the continuity of $F$, the theorem of dominated convergence and the fact that GMC measure converges, we can conclude the convergence as $\epsilon$ tends to $0$.

      For the conditional law part, this follows directly from the convergence of the GMC measure to a measurable function of $\Psi$ (see for example Lemma B.1 of \cite{PS}).
\end{proof} 
\begin{remark}
    \label{r.discontinuous} Note that as the only discontinuity of $\Cov((x,t), \cdot)$ is at the point $(x,t)$. Then, for any deterministic $s\in \R$, the thickness of the point $(x,s)$ is $0$. Furthermore, for a given point $x$, one can also check that the thickness function is completly discontinuous with respect to $t$, as the SHEF restricted to the line $\{(x,s): s\in \R\}$ is a log-correlated field. 
\end{remark}

We now proceed to prove Proposition \ref{MainGMCTP}.

\begin{proof}[Proof of Proposition \ref{MainGMCTP}]
    From Proposition \ref{p.SHEconv}, we know that sampling first $\Psi$ from its marginal in $\mbb Q^\gamma$ and then $(x,t)$ from the normalized $M^{\gamma}(\Psi, dxdt)$ is the same as if we first sample $(x,t)\sim \text{Unif}(D_T)$ and then definining $\Psi= \tilde \Psi+\gamma \Cov((x,t),\cdot)$ where $\tilde \Psi$ is a SHEF under $\mbb{P}$. By evaluating the circle average in $\Psi$, it follows that
    \begin{equation*}
        \frac{\pe{\epsilon}(x,t)}{\ln(\frac{1}{\epsilon})}
        \overset{\text{law}}{=}
        \frac{\widetilde{\Psi}_{\epsilon}(x,t)}{\ln(\frac{1}{\epsilon})} +
        \g \frac{\Var(\pe{\epsilon}(x,t))}{\ln(\frac{1}{\epsilon})}
    \end{equation*}
    for all $\epsilon>0$. Since we know that $\Var(\pe{\epsilon}(x,t)) = \ln(\frac{1}{\epsilon}) + O(1)$, we can conclude that almost surely $(x,t)$ is $\g$-thick point. As the marginal law of $\mbb Q^\gamma$ is are absolutely continuous with respect to that of a SHEF, we conclude.

\end{proof}

From the connection seen in this section, we can obtain geometric information related to $T^{\g}$.

\subsection{Lower bound for the Hausdorff dimension}\label{ss.lowerboundshef}

In this section, we use the results from the previous section to prove the following Proposition.
\begin{prop}
    \label{p.LowerBoundI}
    For $\g \in (0, 2\sqrt{2})$, we have that, almost surely,
    \begin{equation*}
        \min\left \{3-\frac{\g^2}{4}, 4-\frac{\g^2}{2}\right \} \leq \dim_H(T^{\g}).
    \end{equation*}
\end{prop}

We prove this proposition using techniques from potential theory and the connection between the thick points and the GMC. To this end, we introduce the $\alpha$-energy for $\alpha>0$. This is the functional on $\mcl P(\R^d)$ defined by 
\begin{equation}
    \label{eq.s-energy}
    \mu  \mapsto I_\alpha(\mu)= \int_{D_T}\int_{D_T}\frac{\mu(dxdt)\mu(dyds)}{|(x,t)-(y,s)|^\alpha} \in [0, \infty].
\end{equation}
A connection between the Hausdorff dimension of a set and the $s$-energy is what in the literature is called an energy method. This appears, for example, in Proposition 4.13 from \cite{fracgeom}.
\begin{prop}[Proposition 4.13 item $(a)$ from \cite{fracgeom}]
\label{p.fracgeom4.13}
     Let $F$ be a subset of  $\R^d$, assume there is a probability measure $\mu$ supported on $F$ such that $I_{\alpha}(\mu)$ is finite. Then $\dim_H(F)\geq \alpha$.
\end{prop}

 \begin{remark} By Theorem 4.1.2 of \cite{Potentialtheory} we have that $I_{\alpha}$ is a lower semi-continuous function from the space of finite measures equipped with the weak topology since the kernel $k_{\alpha}(x,y) = |x-y|^{-\alpha}$ is also a symmetric lower semi-continuous function from $D_T\times D_T$ to $[0,\infty]$.
 \label{rem.lsc_of_s-energy}
 \end{remark}

In order to use Proposition \ref{p.fracgeom4.13}, we take in account Remark \ref{rem.lsc_of_s-energy} and take a proper approximation of $M^{\g}$ since it is supported on the thick points of the SHEF. To this end, we take $\epsilon_0\in (0,1)$ and $\epsilon \in (0,\epsilon_0)$ and consider for a given $\gamma\in$ $\g\in(0,\sqrt{8})$ and $\rho>\g$ close to $\g$. Then, it actually suffices to consider the limit measure as $\epsilon\to 0$ of
\begin{equation}
\label{equation.shef}
    M^{\g}_{\epsilon_0}(\Psi_{\epsilon},dxdt):= e^{\g \Psi_{\epsilon}(x,t)-\frac{\g^2}{2}\E\left[\Psi^2_{\epsilon}(x,t)\right]}\1_{G_{\rho, \epsilon,\epsilon_0}(x,t)}dxdt.,
\end{equation}
where $G_{\rho, \epsilon,\epsilon_0}(x)$ is given by \eqref{eq.A_GP}. We know that $M^{\g}_{\epsilon_0}(\Psi_{\epsilon},dxdt)$ goes to $ M^{\g}_{\epsilon_0}(\Psi,dxdt)$ as $\epsilon$ tends to $0$. The limit holds almost surely in the weak topology of measures up to a deterministic subsequence (call it $(\epsilon_k)_{k\in\N}$). Notice that a.s. $M^{\g}_{\epsilon_0}(\Psi,dxdt)$ is absolutely continuous with respect to $M^{\g}(\Psi,dxdt)$ (since its Radon-Nykodon derivative is $\1_{G_{\alpha,0, \epsilon_0}}$). Therefore, we can focus on $M^{\g}_{\epsilon_0}(\Psi,dxdt)$ to obtain a lower bound for the fractal dimension, as it is also supported on the $\g$-thick points of the field.
By Remark \ref{rem.lsc_of_s-energy} for every given $\epsilon_0$ it holds almost surely that 
\begin{equation}
\label{eq.s-energiesGMCmeasures-II}
I_{\alpha}(M_{\epsilon_0}^{\g}) \leq \liminf_{k\goesto\infty} I_{\alpha}(M^{\g}_{\epsilon_0}(\Psi_{\epsilon_k})).
\end{equation}
We now compute the lower bound by considering the above discussion.

\begin{proof}[Proof of Proposition \ref{p.LowerBoundI}]
    Let us, first, prove that liminf of  $I_{\alpha}(M^{\g}_{\epsilon_0}(\Psi_{\epsilon_k}))$ is a.s. finite as $k$ diverges while $\epsilon_0$ is fixed. We do this by checking that it belongs to $L^1$ as by Fatou's lemma we have that
    \begin{align*}
        \E\left[\liminf_{k\goesto\infty} 
        I_{\alpha}(M^{\g}_{\epsilon_0}(\Psi_{\epsilon_k}))\right]
        \leq 
        \liminf_{ k\goesto\infty}
        \E\left[ I_{\alpha}(M^{\g}_{\epsilon_0}(\Psi_{\epsilon_k}))\right].
    \end{align*}
    We now claim that the right-hand side is upper bounded.
    \begin{claim}
    \label{c.alphaenergy}
        For any $\epsilon_0>0$ and $\alpha<\min  \left\{3-\gamma^2/4, 4-\gamma^2/2 \right\}$, then 
        \begin{align*}
        \liminf_{ k\goesto\infty}
        \E\left[ I_{\alpha}(M^{\g}_{\epsilon_0}(\Psi_{\epsilon_k}))\right]<\infty.
        \end{align*}
    \end{claim}
The proof of the proposition follows directly from the claim as follows. Choose a random $\epsilon_0>0$ such that $M_{\epsilon_0}^\gamma$ is a non-zero measure. This can be done as $M_{\epsilon_0}^\gamma\nearrow M^\gamma$ and a.s. $M^\gamma\neq 0$. As $M_{\epsilon_0}^\gamma$ is supported on the $\gamma$-thick points and $I_\alpha(M_{\epsilon_0}^\gamma)<\infty$, we conclude by Proposition \ref{p.fracgeom4.13}.

    We are now just missing the proof of the claim.
    \begin{proof}[Proof of Claim \ref{c.alphaenergy}]
    From \eqref{equation.shef} and \eqref{eq.A_GP} respectively, we have then,
    \begin{align}\label{e.EIalpha}
         \E\left[ I_{\alpha}(M^{\g}_{\epsilon_0}(\Psi_{\epsilon_k}))\right]= 
        \int_{D_T}\int_{D_T}\frac{e^{\g^2\Cov(\pe{\epsilon_k}(x,t),\pe{\epsilon_k}(y,s))}
        \mbb{\tilde{P}}({G_{\rho, \epsilon_k,\epsilon_0}(x,t),G_{\g}(y,s)})}
        {|(x,t)-(y,s)|^{\alpha}},
    \end{align}
    and as we saw earlier, $\mbb{\tilde{P}}$ is the measure from Lemma \ref{l.EL}. Using Proposition \ref{CG1P3.1} and \eqref{eq.001} we get the that \eqref{e.EIalpha} is upper bounded, up to a multiplicative constant, by
    \begin{equation}
    \label{eq:001}
        \int\int_{D_T\times D_T}\frac{1}{|(x,t)-(y,s)|^{\alpha}}\frac{1}
        {(|x-y|\lor \sqrt{|t-s|})^{\g^2-\frac{(2\g-\rho)^2}{2}}}.
    \end{equation}
    Since $\rho$ is arbitrarily close to $\g$ we have that $\g^2-\frac{(2\g-\rho)^2}{2} =\frac{\g^2}{2}+\delta $ for some $\delta >0$ that can be arbitrarily small. 
    Let $R$ be the diameter of $D_T$, and take $(y,s) \in D_T$. Arguing as in the proof of Proposition \ref{p.hII}, we know it suffices to focus on
    \begin{equation*}
        \int_{0}^{R}\int_{0}^{R}\frac{1}{(l\lor u)^{\alpha}}
        \frac{1}{(l\lor \sqrt{u})^{\frac{\g^2}{2}+\delta}}ldldu
        \leq
        C(R)
        \int_{0}^{1}\int_{0}^{1}\frac{1}{(r\lor s)^{\alpha}}
        \frac{1}{(r\lor \sqrt{s})^{\frac{\g^2}{2}+\delta} }rdrds.
    \end{equation*}
    Here we have done the change of variables $l=Rr$ and $u=Rs$ and $C(R)>0$ is a constant depending only on $R$. Since $r,s\in [0,1]$ we can split into the following cases: $\{r\leq s\}$,
    $\{s\leq r\leq \sqrt{s}\}$, and $\{\sqrt{s}\leq r\}$. Let us call this separations first, second and third case respectively and study them separately.
    
    \textit{1$^{st}$ case:} We start by checking the case $r\leq s$, here, the integral will be 
    \begin{equation*}
        \int_{0}^{1}\int_{r}^{1}\frac{1}{s^{\alpha}}
        \frac{1}{(\sqrt{s})^{\frac{\g^2}{2}+\delta}}dsrdr
        =
        \int_{0}^{1}\int_{r}^{1}\frac{1}{s^{\alpha+\frac{\g^2}{4}+\frac{\delta}{2}}}
        dsrdr
        =
        \frac{1}{\alpha+\frac{\g^2}{4}+\frac{\delta}{2}-2}
        \int_{0}^1  \frac{1}{r^{\alpha+\frac{\g^2}{4}+\frac{\delta}{2}-2}} - 1.
    \end{equation*}
    Since we want the right-hand side integral to be finite, we need that
    \begin{equation*}
        \alpha< 3 - \frac{\g^2}{4}- \frac{\delta}{2}.
    \end{equation*}
    Since $\delta$ can be chosen arbitrarily small, we end up with 
    \begin{equation*}
        \alpha< 3 - \frac{\g^2}{4}.
    \end{equation*}
    This first step gave us a first value, let us check now the second interval.

    \textit{$2^{nd}$ Case:} Now we see the case $s\leq r\leq \sqrt{s}$. Here we have that the integral is equal to 
    \begin{align*}
        \int_{0}^{1}\frac{1}{(\sqrt{s})^{\frac{\g^2}{2}+\delta}}
        \int_{s}^{\sqrt{s}}\frac{r}{r^{\alpha}}
        drds = 
        \int_{0}^{1}\frac{1}{s^{\frac{1}{2}(\alpha-2+\frac{\g^2}{2}+\delta)}}
        -\frac{1}{s^{(\alpha-2+\frac{\g^2}{4}+\delta)}} ds
    \end{align*}
    and then we have the following conditions for $\alpha$ follows,
    \begin{align*}
        \alpha-2+\frac{\g^2}{4}+\frac{\delta}{2} < 1, \hspace{0.3cm} \frac{1}{2}\left (\alpha-2+\frac{\g^2}{2} +\delta \right ) <1.
    \end{align*}
    This implies that 
    \begin{equation*}
        \alpha < \min\left \{3-\frac{\g^2}{4}, 4-\frac{\g^2}{2}\right\}-\frac{\delta}{2}.
    \end{equation*}
    And again, since $\delta$ can be arbitrarily small, we have the respective values.

    \textit{$3^{rd}$ Case:} And now we finalize checking the case $\sqrt{s}<r$, in this case we have,
    \begin{equation*}
        \int_{0}^{1}\int_{\sqrt{s}}^{1} \frac{1}{r^{\alpha+\frac{\g^2}{2}+\delta-1}}drds
    =
        \frac{1}{\alpha+ \frac{\g^2}{2}+\delta-2}\int_{0}^{1}\frac{1}{s^{\frac{1}{2}(\alpha+\frac{\g^2}{2}+\delta-2)}} - 1ds.
    \end{equation*}
    And therefore we need that
    \begin{equation*}
        \alpha < 4-\frac{\g^2}{2}-\frac{\delta}{2},
    \end{equation*}
    for arbitrarily small $\delta$. Since we need the 3 cases to be finite, we have that 
    \begin{equation*}
        \alpha < \min\left \{3-\frac{\g^2}{4}, 4-\frac{\g^2}{2}\right \}.
    \end{equation*}
 \end{proof} 
\end{proof}
Last but not least, if in $D_T$ we put the metric given by 
$d_p$, and perform the same calculations as before, we have that, we only need steps 1 and 3. Under a small change in the first one, we can conclude that,
\begin{equation*}
    \label{O1}
    \dim_{H,d_p}(T^{\g}) \geq 4-\frac{\g^2}{2},
    \tag{Ob}
\end{equation*}
with this, we can interpret that the change of behaviour is due to the fact that the natural metric for this field is the last one and not the Euclidean one.

\subsection{Upper bound for the Hausdorff dimension}\label{ss.upperbound}
In this section, we will prove the following proposition.
\begin{prop}
    \label{p.UpperBoundSec5}
    For $\gamma \in (0,2\sqrt{2})$, we have that
    \begin{equation}
    \label{eq.shefupbound}
        \dim_H(T^{\gamma}) \leq \min \left \{ 4-\frac{\gamma^2}{2}, 3- \frac{\gamma^{2}}{4}\right \}
    \end{equation}
    and if $\gamma > 2\sqrt{2}$ then, $T^{\gamma}$ is empty.
\end{prop}
In particular, this together with Proposition \ref{p.LowerBoundI}  implies that \eqref{eq.shefupbound} is in fact an equality. To achieve this, we first need to simplify the limit in the definition of thickness as in \cite{HMP}. This is done in the following lemma.
\begin{lemma}
    There exists a deterministic discrete sequence $r_n$ with $r_n \goesto 0$ 
    as $n$ goes to infinity, such that, almost surely, for all $(x,t)\in D_T$, we have 
    \begin{equation*}
        \limsup_{\epsilon\goesto 0}\frac{\pe{\epsilon}(x,t)}{\ln(\frac{1}{\epsilon})} = 
        \limsup_{n \goesto \infty}\frac{\pe{r_n}(x,t)}{\ln(\frac{1}{r_n})}
    \end{equation*}
    \label{SimpTP}
\end{lemma}
\begin{proof}
    The proof is the same line of arguments as in Lemma \ref{disclOU} under small changes.
\end{proof}
From this lemma we can prove the respective upper bound for $T^{\g}$. The main idea in the proof is to consider 2 distinct covers. The first one, given by irregular boxes of length-side at scales $r\times r\times r^2$ will provide a good cover for $\g \in (2,2\sqrt{2})$ and the approximation we need to do for thick point is of order $r$. However for the case $\g\in [0,2]$, it is better to cover with boxes of length-side $r\times r\times r$, but in this case one needs to approximate the thick points at radius $\sqrt{r}$. The reason for these different approximating radius is Proposition \ref{p.ContSHE}.

\begin{proof}[Proof of Proposition \ref{p.UpperBoundSec5}]
    We first considering
    \begin{equation*}
        T^{\geq\g}:= \left\{ (x,t)\in D_T: \limsup_{\epsilon\goesto 0}\frac{\pe{\epsilon}}{\ln\left (\frac{1}{\epsilon}\right )}\geq \g \hs \right\}.
    \end{equation*}
    focus on $T^{\geq\g}$, since $T^{\g}\subset T^{\geq\g}$, it suffices to construct a suitable cover of the latter set. To this end, we consider $r_n$ and $\hat{r}_n = r_n^{1+\delta}$ for some $\delta>0$ as in in Lemma \ref{SimpTP}.
    
    We consider two different covers of $T^{\geq\g}$. The first is given by the boxes whose edges are in $\left(\hat{r}_n\Z^2\right) \times\left(\hat{r}_n^2\Z\right)$ and intersects $D_T$. Let $\alpha \in (0, 1/2)$ and $\zeta\in (0,1)$. By Proposition \ref{p.ContSHE}, a.s. there exists a constant $C>0$, such that, uniformly over $D_T$, for $(x,t),(y,s)\in \square$, with $\square \in \mcl Q^1_n$, we have that
    \begin{align*}
        |\pe{r_n}(x,t) -\pe{r_n}(y,s)| 
        \leq
        C \ln^{\zeta}\left(\frac{1}{r_n}\right) \frac{d_p((x,t),(y,s))^{\alpha}}
        {\hat{r}_n^{\alpha}}
        \leq 
        \sqrt{2}C\ln^{\zeta}\left(\frac{1}{r_n}\right).
    \end{align*}
    This implies
    \begin{equation}
    \label{eq.shef_thick-approximation}
        \pe{r_n}(y,s) \geq \pe{r_n}(x,t) - C\ln^{\zeta}\left( \frac{1}{r_n}\right).
    \end{equation}
    We now argue in the same way as in Section 3 from \cite{HMP} and consider $\mcl I_n$ as the set of boxes in $\mcl Q^1_n$ whose centers satisfy \eqref{eq.shef_thick-approximation}. For a given box $\square \in \mcl Q^1_n$ we estimate the probability that it center satisfies \eqref{eq.shef_thick-approximation}, by using the exponential tails of a normal distribution we have that
     
    \begin{align*}
        \mbb P \left( \square \in \mcl I_n \right) &= 
        \mbb{P}\left(\pe{r_n}(x_{\square},t_{\square}) 
        \geq
        \left(\g -C\ln\left(\frac{1}{r_n}\right)^{\zeta-1}\right)\ln\left(\frac{1}{r_n}\right) \right)\\
        & \leq
        C \exp\left[
        -\frac{1}{2}\ln\left(\frac{1}{r_n}\right)
        \left(\g^2 + C\ln\left(\frac{1}{r_n}\right)^{2(\zeta-1)+o(1)}\right)
        \right],
    \end{align*}
    Taking $\zeta < 0.5$, we have then
    \begin{align*}
        \mbb P \left( \square\in \mcl I_{n} \right)
        \leq
        C \exp\left(-\frac{1}{2}\ln\left(\frac{1}{r_n}\right)
        \g^2+o(1)\right).
    \end{align*}
    Since, $|\mcl{Q}_n^{1}| = O((1/\hat{r}_n)^4)$ we can argue as in Lemma 3.1 of \cite{HMP} to conclude, on one hand that if $\g^2>8$ then the set $T^{\geq\g}$ is empty. On the other hand that for $\g\in (0,\sqrt{8})$ we have that 
    $\dim_H(T^{\geq\g})\leq 4-\g^2/2$.

    We now proceed with the second cover $\mcl{Q}_{n}^{2}$ as the set of cubes with edges in $\left(\hat{r}_n\Z^2\right) \times\left(\hat{r}_n\Z\right)$ that intersects $D_T$. In this case we consider $\pe{\sqrt{r_n}}$ and therefore \eqref{eq.shef_thick-approximation} changes to
    \begin{equation}
    \label{eq.shef_thick-approximation_II}
        \pe{\sqrt{r}_n}(x_{n_k},t_{n_k}) 
        \geq  \pe{\sqrt{r}_n} (x,t)
        - C \ln^{\zeta}\left(\frac{1}{\sqrt{r}_n}\right).
    \end{equation}
    Therefore, we take $\mcl J_n$ as the family of boxes in $\mcl Q^2_n$ whose centers now fulfils \eqref{eq.shef_thick-approximation_II}. Then, for $\square\in \mcl Q^2_n$ we estimate the probability of $\square \in \mcl J_n$ to obtain
    \begin{align*} 
        \mbb{P}\left(\pe{\sqrt{r}_n}\geq \left(\g-C\ln\left(\frac{1}{\sqrt{r}_n}\right)^{\zeta-1}\right)
        \ln\left((\sqrt{r}_n)^{-1}\right)+0(1)\right) 
        \leq
        C \exp(-\frac{\g^2}{4}\ln(\frac{1}{r_n})).
    \end{align*}
    From this we conclude in particular that now the set should be empty for $\g>\sqrt{12}$ and 
    \begin{equation*}
        \dim_H(T^{\geq\g})\leq 3- \frac{\g^2}{4}.
    \end{equation*}
    Hence, the proposition is proved.
\end{proof}
From this section, we can conclude the following proposition.
\begin{prop}
 For each $\g\in(0,2\sqrt{2})$, the set of $\g$-thick points $T^{\g}$ of the SHEF $\Psi$ has Hausdorff dimension
\begin{equation}
    \label{dimshe1sec5}
    \dim_H(T^{\g}) = \min\left \{4-\frac{\g^2}{2},3-\frac{\g^2}{4}\right \}.
\end{equation}
\label{dimsheprop}
\end{prop}
\begin{proof}
    The proof is a direct consequence of Propositions \ref{p.UpperBoundSec5} and \ref{p.LowerBoundI}.
\end{proof}

We can re-write the dimension of $T^{\g}$ as a function $\g$ as follows,
\begin{equation}
    \label{DimSHEsec5}
    \dim_H(T^{\g}) = 
    \begin{cases}
        3-\frac{\g^2}{4} & \text{ for } \g \in [0,2], \\ 
        4-\frac{\g^2}{2} & \text{ for } \g \in (2, 2\sqrt{2}).
    \end{cases}
\end{equation}
We now plot \eqref{DimSHEsec5} with respect to $\g^2$ to visualize the change on the  regularity of $\dim_H (T^{\g})$ at $\g^2=4$
\begin{figure}[h]
    \centering
    \includegraphics[scale=0.48]{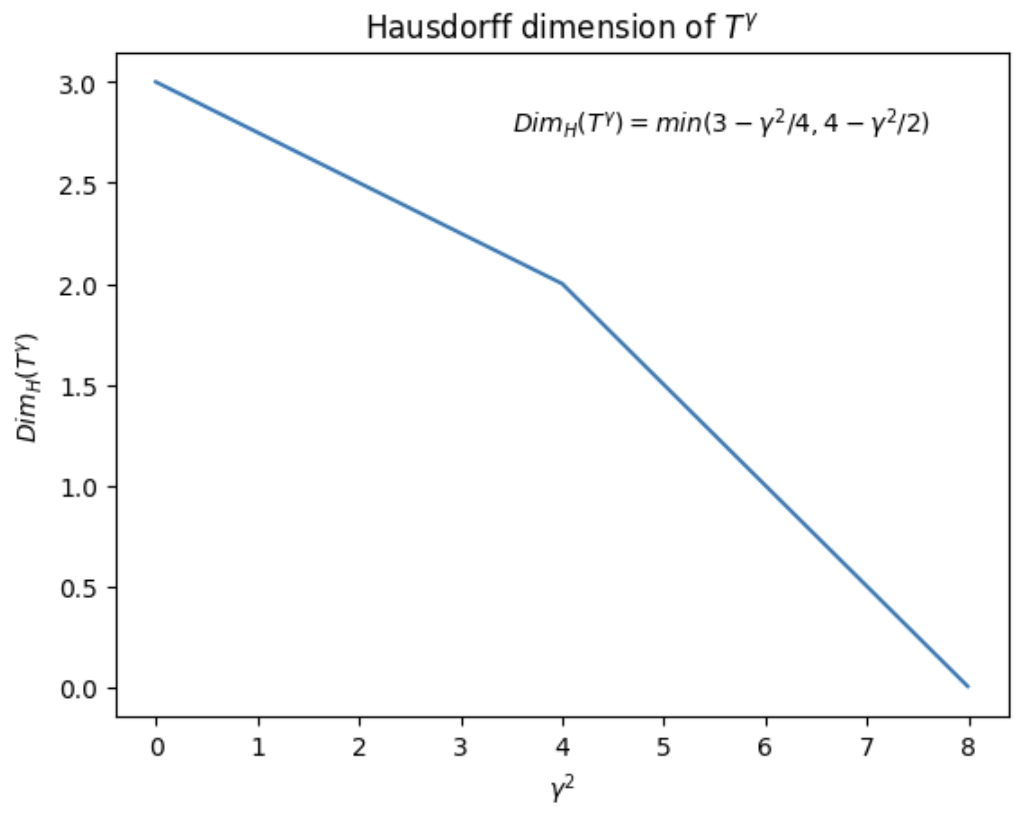}
    \caption{Hausdorff dimension of $T^{\g}$ as a function of $\g^2$.}
    \label{fig.dim-T^g}
\end{figure}
From the graph one can see that the function is no longer differentiable. $\g^2=4$. Moreover, two regimes can be distinguished. The first, which is the natural regime for the GFF, is given by $[0,2]$. We call this regime ``thick-regime'' or ``thick points''. The second regime is given by $(2,2\sqrt{2})$. We refer to this regime as ``super-thick''. This because there are no points in the GFF with this ``thickness''. 

The separation of the regimes ``thick'' and ``super-thick'' is motivated by 2 observations.
The first one is that since $T^{\g}$ has non-trivial Hausdorff dimension, there have to exist random times with super-thick points which implies that at those times the model is not absolutely continuous with respect to the GFF.

The second one is that, one can interpret the lost of differentiability in \eqref{DimSHEsec5} at $\g=2$ a ``phase transition'' on the parameter $\gamma$. Therefore, we can expect a change in the behavior of the thick points between the 2 mentioned regimes. These 2 observations motivate the last two sections.

Finally, let us discuss about what would happen if instead of the Euclidean metric we used the parabolic one.

\begin{rem} 
The above result is valid when we consider $D_T$ under the Euclidean distance. However, the natural scaling of the SHEF is anisotropic. If we instead equip $D_T$ with the parabolic metric $d_p$, the scaling calculations are simpler. Because the space-time domain has a parabolic Hausdorff dimension of $4$, the dimension of the thick points under $d_p$ is simply
\begin{equation*}
    \dim_{H, d_p}(T_{d_p}^{\g}) = 4-\frac{\g^2}{2},
\end{equation*}
so there is no phase transition at $\gamma=2$.

A Liouville measure for a space-time field similar to the SHEF was already studied by Garban in \cite{CG1} and improved by Oh, Roberts and Wang in \cite{ORW}. However, the focus of that work was not on the geometry of the thick points, but rather on the behaviour of the solution to the equation. In that context, the solution exhibits a phenomenon of intermittency, meaning that its Besov regularity is not uniform across all points. This can be explained by the multifractality of the thick points of the field.

\end{rem}

\subsection{Exceptional times}
Since $T^{\g}$ has a positive Hausdorff dimension in the super-thick regime.
We can conclude that, in particular, exists certain times $\Tt$,
where $\Psi(\Tt)$ has super-thick points. This simple observation motivates us to 
talk about this set of ``exceptional times'' and study some aspects of its geometry.

The notion of exceptional times is not new and it has already been seen in dynamical percolation, see \cite{ET-I}, \cite{ET-II} and \cite{ET-III} between other works. It also has been studied the notion of singular time for SPDE's in \cite{AA} although the notion of singular times and exceptional times are not the same, both concepts focus on times where the SPDE has a special behavior of interest.

To put the above idea into a formality. Let $\pi_{\hat{t}}$ be
the projection $(x,t)\in D_T\mapsto t\in [0,T]$. The set of exceptional times is defined as follows
\begin{defn}
    \label{defSTg}
    For $\g\in (2,2\sqrt{2})$. The set $\ET^{\g}$ of exceptional times is defined as
    \begin{equation*}
        \ET^{\g} = \pi_{\hat{t}}(T^{\g}).
    \end{equation*}
\end{defn}
A first natural question is what happens with the intersection given by
$\mbb{Q}\cap \ET^{\g}$. Since that if we take a $q\in \mbb{Q}$, we will have that
\begin{equation*}
    \mbb{P}(q\in \ET^{\g}) =\mbb{P}\left (\exists x\in D, \text{ s.t. } \limsup_{n\goesto\infty}
    \frac{\pe{r_n}(x,q)}{\ln(\frac{1}{r_n})}= \g\right  )
\end{equation*}
but since at time $q$ we know that $\Psi(q)$ has the law of a GFF, this last probability is 
equal to 0. Therefore, we can conclude that almost surely $\mbb{Q}\cap \ET^{\g} = \emptyset$.
From this observation, a first simple lemma about this set would be the next one
\begin{lemma}
    Almost surely, the set of exceptional time is totally disconnect.
\end{lemma}
Now, we want to study a more complex geometric notion of $\ET^\gamma$ as is its Hausdorff dimension and we put as a goal to obtain the fractal dimension of the set  $\ET^{\g}$. In particular, we show the following proposition
\begin{prop}
    \label{p.dimSTI}
    For $\g\in(2,2\sqrt{2})$. Almost surely the Hausdorff dimension of the exceptional times is
    \begin{equation*}
        \dim_H(\ET^{\g}) = 2-\frac{\g^2}{4}.
    \end{equation*}
\end{prop}
Thanks to the work we have done, this result is rather simple to obtain. To start, let us check the upper bound. To do this we recall the next lemma from \cite{PMBM} chapter 4, on the context of metric spaces.
\begin{lemma}
    \label{lhs}
    Let $(X_1,d_1)$ and $(X_2,d_2)$ be two complete and locally compact metric spaces, and consider $f:(X_1,d_1)\goesto(X_2,d_2)$ a surjective map for which there exists $C>0$ and $\alpha>0$  such that for all $x,y \in X_1$, we have that
    \begin{equation*}
        d_2(f(x),f(y)) \leq Cd_1(x,y)^{\alpha}.
    \end{equation*}
    Then, we have that 
    \begin{equation*}
        \dim_H(X_2)\leq \frac{1}{\alpha} \dim_H(X_1).
    \end{equation*}
\end{lemma}
The proof of this lemma can be found in \cite{PMBM} Chapter 4. Using this result, we can prove the upper bound on the dimension of the exceptional times.
\begin{lemma}
    \label{l.upperboundSTI}
    For $\g\in (2,2\sqrt{2})$, we  have that
    \begin{equation*}
        \dim_H(\ET^{\g}) \leq 2-\frac{\g^2}{4}.
    \end{equation*}
\end{lemma}
\begin{proof}[Proof]
    To prove this lemma, let us consider the following map.
    \begin{align*}
        \pi_{\hat{t}}:(D_T, d_p)&\longrightarrow ([0,T], |\cdot|) \\
        (x,t) & \longmapsto t.
    \end{align*}
    Here notice that it is direct that this map is surjective. On the other hand, since in the domain we considered the parabolic metric. We can see that,
    \begin{align*}
        |\pi_{\hat{t}}(x,t) - \pi_{\hat{t}}(y,s)| = |t-s| 
         \leq |t-s| \lor |x-y|^2  
         = d((x,t),(y,s))^2.
    \end{align*}
    Therefore, the map $\pi_{\hat{t}}$ is a function that fulfils the hypothesis of Lemma \ref{lhs} for $\alpha =2$, and hence, if we consider the map  $\pi_{\hat{t}}|_{T_{d_p}^{\g}}$, it also fulfils requirements. 
    Since $\pi_{\hat{t}}(T_{d_p}^{\g}) = \ET^{\g}$. We have then 
    \begin{align*}
        \dim_H(\ET^{\g})&\leq \frac{1}{2}\dim_H(T^{\g})\\
        & \leq \frac{1}{2}(4-\frac{\g^2}{2})
    \end{align*}
    and hence, the lemma is true.
\end{proof}
Now that we have the upper bound for the Hausdorff dimension, we only need to check  the lower bound. To do this, we take the push-forward measures of the chaos measure associated with the field. This is
\begin{equation}
    \label{chaostime1}
    \sigma^{\g}(dt) = (\pi_{\hat{t}})_{\#}(M^{\g}(\Psi,dxdt)).
\end{equation}
From the previous section, we know that this measure is well-defined, and even more, a ``typical time'' for this measure is an exceptional time. We argue as in Section \ref{ss.lowerboundshef} and consider the respective approximation given by
\begin{equation}
    \label{eq.chaostime1}
    \sigma_{\epsilon, \epsilon_0}^{\g}(dt) = (\pi_{\hat{t}})_{\#}(M_{\epsilon_0}^{\g}(\pe{\epsilon},dxdt)),
\end{equation}
and argue as in Proposition \ref{p.LowerBoundI} to obtain the lower bound. This is done in the following lemma
\begin{lemma}
    \label{l.lowerboundSTI}
    For $\g$ in the super-thick regime, we have that
    \begin{equation*}
        \dim_H(\ET^{\g}) \geq 2-\frac{\g^2}{4}.
    \end{equation*}
\end{lemma}
\begin{proof}
We do as in Proposition \ref{p.LowerBoundI} and use Proposition \ref{p.fracgeom4.13} and Remark \ref{rem.lsc_of_s-energy}. Therefore, we consider the respective deterministic  $(\epsilon_k)_{k\in\N}$ and $\epsilon_0$ fixed, to focus on
    \begin{equation}
        \label{energyST1}
        \liminf_{k\goesto\infty}
        \E\left[\int_{0}^{T}\int_{0}^{T}
        \frac{\sigma_{\epsilon_k,\epsilon_0}^{\g}(dt)\sigma_{\epsilon_k,\epsilon_0}^{\g}(ds)}{|t-s|^{\alpha}}\right].
    \end{equation}
    From here, we first recall that, we can choose $\rho>\g$ and close to $\g$ such that
    \begin{equation*}
        \sigma_{\epsilon_k,\epsilon_0}^{\g}(dt) \propto
        \int_D\1_{G_{\epsilon_k,\epsilon_0, \rho}(x,t)}
        e^{\g\pe{\epsilon}(x,t)-\frac{\g^2}{2}\E[\pe{\epsilon}^2(x,t)]}dxdt.
    \end{equation*}
    Therefore, we can conclude that $\E[I_{\alpha}(\sigma_{\epsilon}^{\g})]$ is upper-bounded by
    \begin{equation*}
        \int_{0}^{T}\int_{0}^{T}\frac{1}{|t-s|^{\alpha}}\int_D\int_D
        \frac{dxdydtds}{(\sqrt{|t-s|}\lor|x-y|\lor \epsilon)^{\g^2 +\frac{(2\g-\rho)^2}{2}}}.
    \end{equation*}
    By arguing as in Proposition \ref{p.LowerBoundI} we  can choose $\delta>0$ arbitrarily small and only focus on finding for which values of $\alpha$ the integral
    \begin{equation*}
        \int_0^1\frac{1}{s^{\alpha}}\int_0^1\frac{r}{(r\lor \sqrt{s}
        )^{\g^2/2+\delta}}drds,
    \end{equation*}
    is finite. We separate the case where $r<\sqrt{s}$ and 
    $\sqrt{s}<r$. After the respective computations, we conclude that 
    \begin{equation*}
        \alpha < 2- \frac{\g^2}{4}.
    \end{equation*}
    And from \ref{p.fracgeom4.13} we conclude.
\end{proof}
Notice that with Lemmas \ref{l.upperboundSTI} and \ref{l.lowerboundSTI}, Proposition \ref{p.dimSTI} follows directly.

We saw a first geometric aspect related to this set of exceptional time. The following section is about a change of behavior between the thick regime and the super-thick regime and it is also related to this set of exceptional times.

\subsection{Exceptional times with N thick points}
We saw in \eqref{DimSHEsec5} that at $\g=2$, the dimension seen as a function of the thickness is not differentiable. This motivates us to ask if something is happening after this value. Since the Hausdorff dimension obtained is deterministic, we could think of it as an intrinsic property of the field. Therefore, this ``change of behavior'' of the dimension at $2$ raises the question of whether something is changing in the geometry of the field at the exceptional times. To be more precise, given that the change is at 2, something is happening when we move from the thick regime to the super-thick one. 

The expected change of behavior can be seen in the fibers in space related to the exceptional times. In particular, in the thick regimen, the number of thick points in the fiber would be infinite almost surely. But in the super regime, the fiber of the exceptional time will actually be a finite set. This 
is what we will prove in the following proposition.

First define the fiber of an exceptional time $\Tt \in \ET^{\g}$ as follows

\begin{prop}
\label{p.Ntp-I-nonexistance}
    Take $\g>2$. Almost surely, for any exceptional time $\Tt \in \ET^\gamma$ its associated fiber, $Fib(\Tt):=\pi_{\hat{t}}^{-1}(\{\Tt\})\cap T^{\gamma}$, is finite. Even more, if $N$ satisfies
    \begin{equation}
        \label{MaxN}
        N > \frac{4}{\g^2-4}.
    \end{equation}
    Then, almost surely there are no exceptional times $\Tt$ with at least $N$ thick points in its fiber.
    \label{NmaxProp}
\end{prop}

\begin{proof}
    To prove this proposition, first, fix $\ddd\in (0,1)$ and consider the set
    \begin{equation*}
        T_{\ddd}^{\g,N}:= \left\{
            (x_1,...,x_N,t)\in D^{N}\times [0,T]; \hs
            \forall i\in \{1,...,N\}, \hs (x_i,t)\in T^{\g},\hs d(x_i,x_j)>\ddd\text{ when }
        i\not = j
        \right\}.
    \end{equation*}
    The idea here is to find for which values of $N$ this set is empty.
    We do as before in this paper and follow the ideas from \cite{HMP}. We take $r_n $ as in Lemma \ref{SimpTP} and then take $\hat{r}_n = r_n^{1+\kappa}$ for $\kappa>0$ given. Then we define the family of boxes $\mcl{Q}_n$ with edges in $(\hat{r}_n(\Z^2)^N)\times (\hat{r}_n^2)\Z$ that intersects $D^N\times [0,T]$. For each box $\square \in \mcl Q_n$ we take $\square_i= \pi_{i,t}(\square)$, as the image of $\square\in \mcl Q_n$ under the map $\pi_{i,t}: D^N\times [0,T]\goesto D_T, \,((x_j)_{j=1}^N,t)\mapsto (x_i,t)$, and for each $\square\in \mcl Q_n$ and $i\in\{1,...,N\}$ consider $z_{n_i} = (x_{n_i},t_{n_i})$ as the center of the square $\square_i$.

    By Proposition \ref{p.ContSHE} we know that if $(x,t)\in T^{\g,N}_{\ddd}$,  we can find a box $\square \in \mcl Q_n$ such that their centers satisfies $d(z_{n_i},z_{n_j}) \geq \ddd/2\text{ when } i\not = j$ and also that for every $i =1,...,N$ we have that 
    \begin{equation*} 
        \Psi_{r_n}(x_{n_i},t_{n_i})\geq 
        \left( \g -C\ln\left(\frac{1}{r_n}\right)^{\zeta-1}\right)\ln\left(\frac{1}{r_n}\right).
    \end{equation*}
    We now take $\hat{\Psi}_{r_n}$ as $\pe{r_n}$ divided by $\sqrt{\Var(\Psi_{r_n}(z_{n_i}))}$. Then we get for each $i\in \{1,...,N\}$,
    \begin{equation}
        \label{NormalizedGFF}
        \hat{\Psi}_{r_n}(z_{n_i}) \geq 
        (\g-C\ln(\frac{1}{r_n})^{\zeta-1})\ln^{1/2}(\frac{1}{r_n})+o(1).
    \end{equation}
    Now define $I_n$ as the set of centers from $\mcl Q_n$ that satisfy \eqref{NormalizedGFF}. From here we argue as in Proposition \ref{p.UpperBoundSec5}, and we focus on estimating, for a given $\square\in \mcl{Q}_n$, the probability to be in $I_n$.

    Let $C_n$ be the covariance matrix of the vector $(\hat{\Psi}_{r_n}(z_{n_1}),....,\hat{\Psi}_{r_n}(z_{n_N}))$ we can see that for $i\not = j,\hs |C_n^{i,j}|<1$ even more, as $n$ tends to infinity $C_n$ tends to the identity matrix. This implies that there is an $n_0\in \N$ large enough, such that $\sigma_n^2 =\Var\left(\sum_{i}\hat{\Psi}_{r_n}(z_{n_i})\right) \leq N + O(1)$ for every $n\geq n_0$. Hence, we have that $\mbb P (z_{n}\in I_n)$ is upper bounded by
    \begin{align}
        \mbb P \left(
        \mcl N(0,1)  \geq 
        N(\g-C\ln(\frac{1}{r_n})^{\zeta-1})\ln^{1/2}\left(\frac{1}{r_n}\right)\sigma_n^{-1}
        \right) 
        \leq 
        C\exp\left(
        -\frac{\g^2}{2}N\ln\left(\frac{1}{r_n}\right)
        \right).
    \end{align}
    Since $\# \mcl Q_n =\mcl{O}\left(\left(\frac{1}{r_n}\right)^{2N+2}\right)$, the proposition follows arguing in the same way as in Proposition \ref{p.UpperBoundSec5} to obtain that $T_{\delta}^{\gamma,N}$ is a.s. empty as long as \eqref{MaxN} holds. We conclude by taking $\delta \to 0$ through the rationals.
\end{proof}

As a direct corollary of this proposition, it follows that the upper bound of the Hausdorff dimension for the times with $N$ thick points in its fiber is the following.
\begin{cor}
    \label{NupperBound}
    If $\ET^{\g}_N$ denotes the set of exceptional times with $N$ thick point. If $N$ fulfils \eqref{MaxN}, then
    \begin{equation}\label{e.upper_bound_dim_N}
        \dim_H(\ET^{\g}_N)\leq (N+1) - \frac{N\g^2}{4}.
    \end{equation}
\end{cor}
\begin{proof} The proof has two parts, that are analogue to what we did in Section \ref{ss.upperbound}.
By an analogue argument to that of Proposition \ref{p.UpperBoundSec5} and following Proposition \ref{p.Ntp-I-nonexistance} it can be shown that
\begin{align*}
\dim_{H}(T^{\gamma,N}_\delta) \leq 2N+2-\frac{N \gamma^2 }{2 }. 
\end{align*}
By making the union over all rational $\delta$, we obtain that
\begin{align*}
\dim_{H}(T^{\gamma,N}_0) \leq 2N+2-\frac{N \gamma^2 }{2 }. 
\end{align*}
We can then use Lemma \ref{lhs} together with the ideas of Lemma \ref{l.upperboundSTI} to then show that
\begin{align*}
\dim_{H}(ET^{\gamma}_N) \leq \frac{1 }{2}\left( 2N+2-\frac{N \gamma^2 }{2 }\right ). 
\end{align*}

\end{proof}
In the next section, we start working on the lower bound to match \eqref{e.upper_bound_dim_N}.

\subsection{N-Liouville measure}
In Section \ref{GMCfortheSHEfield}, we constructed the chaos measure of a given SHEF and related this measure to its thick points. In the present section, we construct an analogous measure which allows us to find exceptional times with $N$ thick points in its fiber. We start by considering the following sequence of measures 
\begin{equation}
    \label{NGMCaprox}
    \mu^{\g,N,\delta}_{\epsilon}(d\vec{x}dt)=\exp\left(\g\sum_{i=1}^{N}\pe{\epsilon}(x_i,t) 
    -\frac{\g^2}{2}\E[(\sum_{i=1}^{N}\pe{\epsilon}(x_i,t))^2]\right)\1_{A^\ddd_N} \1_{G_{\alpha,\epsilon}} d\vec{x}dt,
\end{equation}
where $A^\ddd_N$ is defined as 
\begin{equation*}
    A^\ddd_N := \{\vec{x}\in D^N;\hs \text{ for every } i,j =1,...,N, \hs i\not = j,
    \hs d(x_i,x_j)>\ddd \},
\end{equation*} 
and 
\begin{equation}
    \label{e.NGoodpoints}
    G^N_{\alpha,\epsilon}(\vec{x},t) = \bigcap_{i=1}^{N}G_{\epsilon, \alpha}(x_i,t),
\end{equation}
where $G_{\epsilon,\alpha}(x,t)$ is from \eqref{eq.A_GP}. The sequence $(\mu^{\g,N,\delta}_{\epsilon})_{\epsilon}$ converges almost surely. More formally, we have the following proposition.
\begin{prop}
\label{p.NgmcLimit}
    The sequence $\mu^{\g,N,\delta}_{\epsilon}$ converges to a nontrivial measure $\mu^{\g,N,\delta}$ as long as 
    \begin{equation*}
        N\left(\frac{\g^2}{4}-1\right) < 1.
    \end{equation*}
\end{prop}

As a first observation, note that we are not in the setting of Theorem \ref{t.generalgmc} as the lift of $\Psi$ to $D^N \times [0,T]$ is not a pseudo-log correlated field. Nevertheless, for certain values of $\g>2$ and $N\in \N$ one can still establish the convergence of $(\mu_{\epsilon}^{\g,N})_{\epsilon}$. 

As for Theorem \ref{t.generalgmc}, we only need to study integrals with respect to deterministic functions. To do this fix a deterministic $f\in C(\overline{D}^{,N}\times[0,T])$ and $\epsilon_0>0$, and define the family of random variables $(I_{\epsilon})_{\epsilon}$ by
\begin{equation}
\label{eq.N-GMC_obs}
    I_{\epsilon}=\int_{D^N_T}f(\vec{x},t)\1_{G^N_{\alpha,\epsilon}(\vec{x},t)}\mu^{\gamma,N}_{\epsilon}(d\vec{x}dt).
\end{equation}
The only part that is not analogous to that of Theorem \ref{t.generalgmc} is to show that $(I_\epsilon)_{\epsilon>0}$ is uniformly bounded in $L^2$.

\begin{prop}
\label{p.Nconvergence}
    For $\alpha>\g$ close enough to $\g$, the sequence $(I_{\epsilon})_{\epsilon>0}$ is uniformly bounded in $L^2$ if 
    \begin{equation*}
        N\left(\frac{\g^2}{4}-1\right) < 1.
    \end{equation*}
\end{prop}
\begin{proof}
    Let $\epsilon>0$. By using classical computations, similar to those of Proposition \ref{p.mainber}, we can see that by writing $z_1= ((x_i)_{i=1}^N, t)$  and $z_2= ((y_i)_{i=1}^N, s)$
    \begin{align}
        \label{l2momN}
        \E[I^2_{\epsilon}] =
         \int_{D^N_T}\int_{D^N_T}
        \f{1}\f{2}
        e^{\g^2\sum_{i,j=1}^{N}\Cov(\pe{\epsilon}(x_i,t),\pe{\epsilon}(y_j,s))}
        \E_{\widetilde{\mbb{P}}}[\1_{G^N_{\alpha,\epsilon}(\vec{x},t)}
        \1_{G^N_{\alpha,\epsilon}(\vec{y},s)}]
        \1_{A^\delta_N}(\vec{x})
        \1_{A^\delta_N}(\vec{y})dz_1dz_2,
    \end{align}
    where 
    \begin{equation*}
        \label{multitilde}
        \frac{d\widetilde{\mbb{P}}}{d\mbb{P}} 
        = 
        \frac{\exp\left(\sum_{i=1}^{N}\g(\pe{\epsilon}(x_i,t)+\pe{\epsilon}(y_i,s))
        \right)}{\E\left[\exp\left(\sum_{i=1}^{N}\g(\pe{\epsilon}(x_i,t)+\pe{\epsilon}(y_i,s))
        \right)\right]}.
    \end{equation*}
    By Proposition  \ref{CG1P3.1} we have that there exists a finite constant $C>0$ such that
    \begin{equation}
        \label{Nupperbound1}
        e^{\g^2\sum_{i,j=1}^{N}\Cov(\pe{\epsilon}(x_i,t),\pe{\epsilon}(y_i,s))}\leq
        C\prod_{i,j=1}^{N}\left(
            \frac{1}{|x_i - y_j|\lor\sqrt{|s-t|}\lor \epsilon}\right)^{\g^2}.
    \end{equation}
    On the other hand we make the following claim
    \begin{claim}
        \label{cl.claimN}
        As $\epsilon\to 0$ 
        \begin{equation}
            \label{upperclaimN}
            \widetilde{\mbb{P}}(G^N_{\alpha,\epsilon}(\vec{x},t),G^N_{\alpha,\epsilon}(\vec{y},s))
            \leq 
            \prod_{i=1}^{N}(r_i)^{\frac{1}{2}(2\g-\alpha)^2+o(1)},
        \end{equation}
        where $r_i = |x_i-y_i|\lor\sqrt{|t-s|}\lor \epsilon$ and we consider 
        $r_i\leq \epsilon_0$ for every $i=1,...,N$, where $o(1)$ is uniform over all compacts of $D^N\times \R$.
    \end{claim}
    From \eqref{Nupperbound1} and \ref{upperclaimN} from the above claim, in \eqref{l2momN} we argue as in \ref{p.LowerBoundI} and consider  $\alpha\approx \g$, this implies that there is exist $\eta>0$ arbitrarily small such that
    \begin{equation}
        \label{l2momN2}
        \E[I^2_{\epsilon}]\leq 
        C \int_{D^{N}_{T}} \int_{D^N_T}\1_{A^\delta_N}(\vec{x})\1_{A^\delta_N}(\vec{y})f(z)f(w)
        \prod_{i,j=1}^{N}\left(
            \frac{1}{|x_i - y_j|\lor\sqrt{|s-t|}\lor \epsilon}\right)^{\frac{\g^2
            }{2}+\eta}dzdw
    \end{equation}
    By noting that for each $x_i$, at most one $y_j$ is closer than $\delta/2$, when both $\vec{x}, \vec{y}\in A^\delta_N$, we see that the \eqref{l2momN2} is upper bounded by a deterministic constant times
    \begin{align*}
    N!\delta^{-N(N-1)(\frac{\gamma^2 }{2 } + \eta)}\int_{D^{N}_{T}} \int_{D^N_T}\1_{A^\delta_N}(\vec{x})\1_{A^\delta_N}(\vec{y})f(z)f(w)
        \prod_{i=1}^{N}\left(
            \frac{1}{|x_i - y_i|\lor\sqrt{|s-t|}\lor \epsilon}\right)^{\frac{\g^2
            }{2}+\eta}dzdw.
    \end{align*}

    From here, we fix $y_i$ and integrate over $x_i$. In this case, by going to polar coordinates the problem can be reduced to showing that the following integral is finite
    \begin{equation}
        \label{Nint}
        I_N = \int_{0}^{1} \int_{[0,1]^N}\prod_{i=1}^{N} 
        \left(\frac{1}{r_i \lor \sqrt{s}}\right)^{\frac{\g^2 }{2}+\eta} r_idr_i ds =
        \int_0^1\prod_{i=1}^{N} \int_{0}^{1}
        \left(\frac{1}{r_i \lor \sqrt{s}}\right)^{\frac{\g^2
            }{2}+\eta} r_idr_i ds.
    \end{equation}
    To know for which values of $\g$ the above integral is finite, we first 
    notice that 
    \begin{align*}
        \int_{0}^{1}
        \left(\frac{1}{r_i \lor \sqrt{s}}\right)^{\frac{\g^2}{2}+\eta} r_idr
        & = 
        \int_{\sqrt{s}}^{1}
        \frac{1}{r_i^{\frac{\g^2}{2}+\eta-1}}dr_i + 
        \int_{0}^{\sqrt{s}}
        \left(\frac{1}{\sqrt{s}}\right)^{\frac{\g^2}{2}+\eta}r_idr_i \\ 
        & \leq 
        -c + \left(\frac{1}{2}+c\right)\frac{1}{s^{\frac{\g^2}{4}+\frac{\eta}{2}-1}}
    \end{align*}
    for $c = (\frac{\g^2}{2}+\eta-2)^{-1}$. When $\g\in (2,2\sqrt{2})$ the above constant is always  positive and finite. From this computation we have then
    \begin{equation*}
        \label{Nint2}
        I_N = \int_{0}^{1} 
        \left(-c + \left(\frac{1}{2}+c\right)\frac{1}{s^{\frac{\g^2}{4}+\frac{\eta}{2}-1}}\right)^Nds. 
    \end{equation*}
    Since $\ddd$ can be arbitrarily small, we need that
    \begin{equation*}
        N\left(\frac{\g^2}{4}-1\right) < 1.
    \end{equation*}
    Therefore, to conclude the proposition we only need to prove the claim 
    \begin{proof}[Proof of Claim \ref{cl.claimN}]
        It is straightforward that for $r_i = |x_i-y_i|\lor\sqrt{|t-s}\lor \epsilon$
        \begin{equation*}
            \widetilde{\mbb{P}}(G^{N}_{\alpha,\epsilon}(\vec{x},t),
            G^{N}_{\alpha,\epsilon}(\vec{y},s)) 
            \leq 
            \widetilde{\mbb{P}}(\forall i =1,...,N,\hs 
            \pe{r_i}(x_i,t)\leq \alpha\ln(1/r_i)).
        \end{equation*}
        Since under $\widetilde{\mbb{P}}$, we have that
        \begin{equation*}
            \pe{r_i}(x_i,t) \Os{\mcl{L}}{=}
            \widetilde{\Psi}_{r_i}(x_i,t) + \g 
            \sum_{j=1}^{N}\Cov\left(\pe{r_i}(x_i,t),\pe{\epsilon}(x_j,t)+\pe{\epsilon}(y_j,s)\right),
        \end{equation*}
        where $\widetilde{\Psi}$ is a SHEF under $\widetilde{\mbb{P}}$. From here consider $r_i=|x_i-y_i|\lor\sqrt{|t-s|}\lor \epsilon$ to obtain the upper-bound given by
        \begin{equation}
            \widetilde{\mbb{P}}\left(\forall i =1,...,N,\hs 
            \widetilde{\Psi}_{r_i}(x_i,t)
            \leq (\alpha-2\g)\ln(1/r_i) +C\right).
        \end{equation}
        Take $\sigma_{i,N}=\sqrt{\Var(\widetilde{\Psi}_{r_i}(x_i,t))}$ and $X_i =-\widetilde{\Psi}_{r_i}(x_i,t)/\sigma_{i,n} $,  we have that the terms above are upper bounded by
        \begin{equation*}
            \widetilde{\mbb{P}}\left(
            \forall i =1,...,N,\hs 
            X_i \geq  (2\g-\alpha)\ln^{\frac{1}{2}}(1/r_i) +C
            \right) \leq \tilde C\inf_{\lambda \in (\R^+)^{N}} e^{\frac{\lambda^t C^n \lambda }{2 } - \sum_{i=1}^N (2\g-\alpha)\lambda_i \sqrt{\log(1/r_i)}},
        \end{equation*}
        where we use the Chernoff bound and defined $C^n$ the covariance matrix of $(X_i)_{i=1}^N$. Noting that $C^n$ converges to the identity uniformly as $n$ goes to infinity, we can take $\lambda_i = (2\g-\alpha)\sqrt{\log(1/r_i)}$ to see that
        \begin{equation*}
            \widetilde{\mbb{P}}(\forall i =1,...,N,\hs 
            X_i
            \geq (2\g-\alpha)\ln^{\frac{1}{2}}(1/r_i) +C) 
            \leq \prod_{i=1}^{N}(r_i)^{\frac{1}{2}(2\g-\alpha)^2 +o(1)}.
        \end{equation*}
        Hence, the claim is proven.
    \end{proof}
\end{proof}

From this proposition, we can swiftly show Proposition \ref{p.NgmcLimit}.
\begin{proof}[Proof of Proposition \ref{p.NgmcLimit}]
    The proof of this proposition follows the same steps as those of Theorem \ref{t.generalgmc} but using Proposition \ref{p.Nconvergence} instead Proposition \ref{p.mainber}. This is because the $L^1$ limit of $\int f\1_{(G^N_{\epsilon,\alpha, \epsilon_0})^c}\mu^{\g,N,\delta}_{\epsilon}$ is $0$, and that the rest of the proof follows as in Theorem \ref{t.generalgmc}.
\end{proof}

\subsection{The SHEF from an N-Liouville typical point}

Now that we have constructed $\mu^{\g,N,\delta}$, we show that its support is $T^{\gamma,N}_\delta$. The present section is an analogous to Section \ref{GMCYTP}, and the main result is the following. 
\begin{prop}
    \label{Ntypic}
    Let $\Psi$ be the SHEF and $\mu^{\g,N,\delta}$ be its N-Liouville measure of parameter $\gamma$. The support of $\mu^{\gamma,N}$ consists of pairs $((x_i)_{i=1}^N,t)$ such that all $x_i$ are different $\gamma$-thick points of $\Psi(\cdot,t)$. 
\end{prop}
We avoid the proofs since are the same as in Section \ref{GMCYTP}, nonetheless we state the necessary results to conclude.

We take the $\epsilon$-rooted measure given by
\begin{equation*}
    \label{NRooted}
    \mbb{Q}_{\epsilon}^{\g,N} := \frac{\exp\left(\g\sum_{i=1}^N \pe{\epsilon}(x_i,t)
   -\frac{\g^2}{2}\E[(\sum_{i=1}^N \pe{\epsilon}(x_i,t))^2] \right)}{\lambda(D^N_T)} 
    \mbb{P}(d\Psi)d\vec{x}dt.
\end{equation*}
From this measure we have the following property. 
\begin{lemma}
    Let $(\vec{x},t,\Psi)$ be sampled from $\mbb{Q}_{\epsilon}^{\g,N}$. The conditional law of  $(\vec{x},t)$ given $\Psi$ is 
    \begin{equation*}
        \nu_{\epsilon}^{\g,N}(d\vec{x}dt):=\frac{\exp\left(\g\sum_{i=1}^N\pe{\epsilon}(x_i,t)
        -\frac{\g^2}{2}\E[(\sum_{i=1}^N\pe{\epsilon}(x_i,t))^2]\right)
        }{\mu_{\epsilon}^{\g,N}(D^N_T)}d\vec{x}dt.
    \end{equation*}
    And on the other hand, the law of $\Psi$ conditionally on $(\vec{x},t)$, is equal to 
    \begin{equation*}
        \hat{\Psi} + \g\sum_{i=1}^{N}\Cov(\cdot, \pe{\epsilon}(x_i,t)),
    \end{equation*}
    where $\hat{\Psi}$ has the law of the SHEF on $Ps(D, [0,T])$.
\end{lemma}
\begin{proof}
The proof is analogous to that of Lemma \ref{l.EL}, in particular, since we still have a Gaussian process, we are able to apply the Cameron-Martin theorem.
\end{proof} 

A direct corollary of the lemma concerns the computation of expected values under the measure $\mbb{Q}_{\epsilon}^{\g,N}$. 
\begin{cor}
    For every $F:Ps(D, [0,T])\times D^N_T \rightarrow \R$, continuous and bounded, we have 
    that 
    \begin{equation*}
        \E_{\mbb{Q}_{\epsilon}^{\g,N}}[F(\Psi,(\vec{x},t))] 
        =
        \E_{\mbb{P}\times \mbb U}
        \left[
        F\left(\Psi +
        \g\sum_{i=1}^{N}\Cov(\cdot, \pe{\epsilon}(x_i,t)), (\vec{x},t)U(d\vec{x}dt),
        \right)
        \right]
    \end{equation*}
    where, $\mbb U$ indicates the uniform distribution on $D^N_T$
\end{cor} 
From this corollary we can easily prove to the following lemma
\begin{lemma}
    The measure $\mbb{Q}_{\epsilon}^{\g,N}$ converges 
    in the weak topology to a measure $\mbb{Q}^{\g,N}$ that fulfils 
    \begin{equation*}
        \E_{\mbb{Q}^{\g,N}}[F(\Psi,(\vec{x},t))] 
        =
        \E_{\mbb{P}\times \mbb U}\left[
        F\left(\Psi +
        \g\sum_{i=1}^{N} \Cov(\cdot, (x_i,t)), (\vec{x},t)
        \right)
        \right],
    \end{equation*}
    where $\Cov(\cdot, (x_i,t))$ is the limit in the Schwartz sense of $\Cov(\cdot, \pe{\epsilon}(x_i,t))$.
\end{lemma}
\begin{proof}
The proof is analogue to Proposition \ref{p.SHEconv}.
\end{proof}

After this work we can easily prove 
Proposition \ref{Ntypic} the same way we prove Proposition \ref{MainGMCTP}.

\subsection{Hausdorff dimension of exceptional times with N points}
The above work related to the N-Liouville measure $\mu^{\g,N,\delta}$ allows us to prove a last result related to the set of exceptional times $\ET^{\g}$ for $\g$ in the super-thick 
regime. This gives the following result. 
\begin{prop}
    \label{NdimT2}
    For $\g\in (2,2\sqrt{2})$, let $\ET^{\g}_N$ be the set of exceptional times
    with $N$ super-thick points at each time, then we have that,  
    \begin{equation*}
        \dim_H(\ET^{\g}_N) = (N+1)-\frac{N\g^2}{4}
    \end{equation*}
\end{prop}
We already have proven the upper bound for this dimension in Corollary \ref{NupperBound}. To obtain the lower bound, we use Proposition \ref{p.fracgeom4.13} and consider the measure
\begin{equation}
    \label{NtimesMeasure}
    \sigma^{\g,N,\delta}(dt) =  (\pi^{N}_{\hat{t}})_{\#}\mu^{\g,N,\delta}(dt).
\end{equation}
Here $\pi^{N}_{\hat{t}}$ is the projection from $D_T^N$ to $[0,T]$. From Proposition \ref{Ntypic} We know that a typical time for $\sigma^{\g,N,\delta}$ is an exceptional time with $N$ thick points. We first fix some positive $\epsilon$ and take 
\begin{align}
\label{eq.NtimesMeasure}
    \sigma^{\g,N,\delta}_{\epsilon,\epsilon_0} =  (\pi^{N}_{\hat{t}})_{\#}\mu^{\g,N,\delta}_{\epsilon,\epsilon_0}(dt)
\end{align}
where $\mu^{\g,N,\delta}_{\epsilon,\epsilon_0}$ is given by \eqref{NGMCaprox}. Since $\sigma^{\g,N,\delta}_{\epsilon,\epsilon_0}$ has a subsequence convergent almost surely to \eqref{NtimesMeasure}, by taking $(\epsilon_k)_{k\in\N}$ as such sequence, we can argue as in the discussion of Proposition \ref{p.LowerBoundI} to use the energy method to obtain the lower bound for $\ET^{\g}_N$.

\begin{proof}[Proof of Proposition \ref{NdimT2}]
    We focus on finding for which values of $\alpha, \hs I_{\alpha }\left(\sigma_{\epsilon_k, \epsilon^l_0}^{\g,N}\right)$ is finite. By \eqref{Nupperbound1} and \eqref{upperclaimN}, we have that, for $\rho$ close to $\g$,
    \begin{align*}
        \E\left[I_{\alpha }\left(\sigma_{\epsilon_k,\epsilon^l_0}^{\g,N}\right)\right] 
        &=
        \int_{0}^{T}\int_{0}^{T} \int_{D^{N}}\int_{D^{N}}
        \E\left[\frac{\1_{G^N_{\rho, \epsilon_k,\epsilon^l_0} (\vec{x},t)G^N_{\rho, \epsilon^l_0}(\vec{y},s)} 
        \mu_{\epsilon_k}^{\g,N}(d\vec{x}dt)\mu_{\epsilon_k}^{\g,N}(d\vec{y}ds)}{|t-s|^{\alpha}}\right] \\
        &\leq C
        \int_{0}^{T}\int_{0}^{T}\int_{D^{N}}\int_{D^{N}}
        \frac{1}{|t-s|^{\alpha}}
        \prod_{i=1}^{N}\frac{1}{(|x_i-y_i|\lor\sqrt{|t-s|}\lor \epsilon
        )^{\frac{\g^2}{2}+\ddd}}.
    \end{align*}
    By arguing as in Proposition \ref{p.hII} we can reduce the problem to check 
    for which values of $\alpha$ the following integral is finite
    \begin{equation*}
        I_N = \int_0^1\int_{[0,1]^N}\frac{1}{s^{\alpha}}
        \prod_{i=1}^N\frac{r_i}{(r_i \lor \sqrt{s})^{\frac{\gamma^2 }{2 }+\delta}}dr_ids.
    \end{equation*}
    Using \eqref{Nint2}, we can see that 
    \begin{equation}
        I_N = \int_{0}^{1}\frac{1}{s^{\alpha}}\left(c_1+c_2\frac{1}{s^{\frac{\g^2}{4}-1}}\right)^{N}ds.
    \end{equation}
    From here we deduce the following condition on $\alpha$
    \begin{equation*}
        \alpha+N(\frac{\g^2}{4}-1)<1,
    \end{equation*}
    or, equivalently
    \begin{equation}
        \alpha <N+1 -\frac{N\g^2}{4}.
    \end{equation}
    Therefore, we can conclude the proposition.
\end{proof}

From here, we can see that, the sequence defined by 
\begin{equation}
\label{Numcriticval}
    \g_N = \sqrt{\frac{4}{N}+4},
\end{equation}
gives us a sequence of phase transitions of this model. Since we have that $\dim_H(\ET^{\g_N}_N)=0$, it is not clear whether for the critical value there exists the respective thick points in space and time. However, we believe that they exist but they should be related to the critical $N$-Liouville measures which we do not attempt to construct in this paper, but it should be done as in 
\cite{DRSV14a,DRSV14b,EP18}
(see \cite{powellcriticgmcrev} for a review on the topic of critical Liouville measure).

\appendix 
\section{GMC measure for pseudo-log correlated fields}
\label{Appen.A}

We now check that the work done in \cite{Ber} can be extended in a more general setting that we call ``pseudo-log-correlated fields''. This can be understood as knowing under which conditions, the following sequence in $\epsilon$ converges almost surely weakly in the space of probability measures,
\begin{equation}
\label{eq.generalGMC}
    \widehat{\mu}_{\epsilon}^{\g}(dx) = \exp\left(\g h_\epsilon(x) - \frac{\g^2}{2}\E\left[(h_\epsilon(x))^2\right]\right)\sigma(dx).
\end{equation}
Sufficient conditions are given in the present theorem.
\begin{thm}\label{t.generalgmc}
    Let $h$ be a Gaussian, such that the following hypotheses are fulfiled 
    \begin{enumerate}
        \item[(H1)]\label{hyp.I} The correlations of $(h_{\epsilon}(x))_{\epsilon\in(0,1), x\in D}$ fulfil
        \begin{equation}\label{eq.lem3.5}
            \Cov(h_{\epsilon}(x),h_{\delta}(y)) = -\ln\left(d(x,y)\lor \epsilon\lor \delta \right) + O(1),
        \end{equation}
        uniformly for $x,y\in D$ and $\epsilon, \delta \in [0,1]$.
        \item[(H2)]\label{hyp.II} There is a $\rho>0$ such that
        \begin{equation*}
            \int\int_{\overline{D\times D}}\frac{\sigma(dx)\sigma(dy)}{d(x,y)^{\rho- \epsilon}} < \infty,
        \end{equation*}
        for every $\epsilon >0$.
    \end{enumerate}
    Then for every $\gamma^2 < 2\rho$, the sequence $\mu_{\epsilon}^{\g}$ defined in \eqref{eq.generalGMC} converges in probability to a non-trivial measure $\mu^\g$.
\end{thm}

We prove this result in $\R^k$ endowed with a metric $d$ that generates the usual topology. We consider a Kernel $K:D\times D\goesto \R$ such that for every $x\not =y\in D$
\begin{equation}
\label{eq.generalKernel}
    K(x,y) = -\ln(d(x,y)) + g(x,y),
\end{equation}
where $g\in C(\overline{D \times D})$. On the other hand we also consider subset of $\mcl M(D)$ (the set of positive measures) such that,
\begin{equation*}
    \mcl M^{+}_{K} (D) := \left\{\mu \in \mcl M (D):  \int\int_{D\times D}|K(x,y)|\mu(dx)\mu(dy) < \infty\right\}.
\end{equation*}
From here we take $\mcl M_{K}(D)$ as the set of signed measures $\rho = \rho_{+} -\rho_{-}$ with $\rho_{\pm}\in \mcl M^{+}_{K} (D)$. Fix $\theta$ a nonnegative Radon measure on $\R^k$ whose support is contained in the unit ball $\overline{B(0,1)}\subset \R^{l}$ for $1\leq l\leq k$ such that its total mass is equal to $1$, and assume that 
\begin{equation}\label{eq.H0}
    \int |\ln(d(x,y)|\theta(dy) \leq C<\infty,
\end{equation}
where $C$ is uniform over $B(0,5)$ (under certain cases one can prove the above, some examples are $|\cdot|$ the usual norm, or norms of the form $|\cdot|^{\alpha}+|\cdot|^{\beta}$ for $\alpha,\beta \in (0,1]$).
We denote by $\theta_{\epsilon,x}$ the measure $\theta_{\epsilon}(\cdot) = \theta(\cdot/\epsilon)$ shifted by $x$ $\theta_{\epsilon,x}(\cdot) = \theta_{\epsilon}(\cdot-x)$ and assume $\theta\in \mcl M_K (D)$. Now consider $(h)_{\sigma \in \mcl M_K(D)}$ a centered Gaussian process with covariance Kernel given by $K$. 
\\

Take a Radon measure $\sigma$ on $\partial D$. We now want to construct the measure whose Radon-Nykodin derivative with respect to $\sigma$ is $\exp(h)$. To do this, we consider $h_{\epsilon}(x) = h(\theta_{\epsilon,x})$, and we will assume that \hyperref[hyp.I]{(H1)} and \hyperref[hyp.II]{(H2)} hold.

Notice that in particular that \hyperref[hyp.I]{(H1)} implies that for each $\epsilon>0$,  $(h_\epsilon (x))_{x\in D}$ is well defined for each given $x\in D$. On the other hand we  observe that (H1) is not always a direct consequence of \eqref{eq.generalKernel} and one might need to prove it case by case. As an example, the GFF with the uniform measure in $B(x,\epsilon)$ fulfils our hypothesis when we consider $\sigma = \lambda$ the Lebesgue measure in $D$.

In \cite{Ber}, one of the ideas that help prove the convergence of the above sequence is by taking in consideration the set of so-called  \textit{good points}. A point $x\in D$ is called a good point if, for a given fix $\epsilon_0>0$, the following holds for every positive $\epsilon< \epsilon_0$,
\begin{equation}\label{eq.A_GP}
        G_{\alpha, \epsilon,\epsilon_0}(x):= \left\{
        h_{r}(x) \leq \alpha\ln(1/r), \forall r\in (\epsilon,\epsilon_0]
        \right\}.
\end{equation}
From now on we fix $\epsilon_0$ and consider $G_{\alpha, \epsilon,\epsilon_0}(x) = G_{\alpha, \epsilon}(x)$. The above set has an important influence on the limit behavior of $\mu^{\g}_{\epsilon}$, since the limit is concentrated (with high probability) in this set when $\alpha$ is close to $\gamma$, this is formalized in the following lemma. 
\begin{lemma}
    \label{l.sec3ber1}
    For $\alpha > \g$, there is $p:\R\goesto [0,1]$ such that 
    \begin{equation*}
        \E[\1_{ G_{\epsilon, \alpha}(z)}e^{\g h_{\epsilon}(x)- 
        \frac{\g^2}{2}\E[h_{\epsilon}^2(x)]}]\geq 1-p(\epsilon_0),
    \end{equation*}
    and $p$ is such that tends to  $0$ as $\epsilon_0\goesto 0$. This implies that 
    \begin{equation*}
        \E\left[\1_{ G^{c}_{\epsilon, \alpha}(z)}e^{\g h_{\epsilon}(z)- 
        \frac{\g^2}{2}\E[h_{\epsilon}^2(z)]}\right] \underset{\epsilon_0\goesto0}{\longrightarrow}
        0.
    \end{equation*}
\end{lemma}
\begin{proof}[Proof]
    The proof of the lemma is the same computation performed in Section 3 of \cite{Ber}.
\end{proof}
Now that we see that the contribution of points that are not good points is small (in expectation), we give the first proposition that will help to prove the convergence.
\begin{prop}
\label{p.mainber}
    Let $h$ be the Gaussian random field with covariance $K$ as in \eqref{eq.generalKernel}. For a given $f\in C\left(\overline{D}\right)$, define 
    \begin{equation*}
        I_{\epsilon,\epsilon_0} = \int_D f(x) \1_{G_{\alpha, \epsilon}(x)} \exp\left(\g h_\epsilon(x) - \frac{\g^2}{2}\E\left[h^2_\epsilon(x)\right]\right)dx.
    \end{equation*}
    If $I_{\epsilon}$ is uniformly bounded in $L^2$ and $\Cov(h_{\epsilon}(x), h_{\delta}(y))$ converges punctually to a function $G(x,y)$ when $\epsilon$ and $\delta$ go to $0$ for $x\not =y$, then, $\mu_{\epsilon}^{\g}$ given by \eqref{eq.generalGMC} converges in probability to a non-trivial random measure $\mu^{\g}$.
\end{prop}
\begin{proof}
    We argue as in Section 4 of \cite{Ber}. We first check that the random variable $(I_{\epsilon})_{\epsilon\in(0,1)}$ is Cauchy in $L^2$. To achieve this, we first notice that for $\epsilon, \delta$ given, we have that, after applying Fubini's theorem, $\E\left[I_{\epsilon} I_{\delta}\right]$ is equal to
    \begin{equation}
    \label{eq.generalcorr}
        \int_D\int_D f(x)f(y)e^{\g^2 \Cov(h_{\epsilon} (x), h_{\delta}(y)}\widetilde{\mbb P}\left(G_{\epsilon, \alpha}(x),G_{\delta, \alpha}(y)\right)dxdy,
    \end{equation}
    where $\widetilde{\mbb P}$ is the measure whose Radon-Nikodyn derivative with respect to $\mbb P$ is 
    \begin{equation}
    \label{eq.generalchangemeas}
        \frac{d \widetilde{\mbb P}}{d \mbb P}(h) = \frac{e^{\g(h_{\epsilon}(x)+h_{\delta}(y)}}{\E\left[e^{\g(h_{\epsilon}(x)+h_{\delta}(y)}\right]}.
    \end{equation}
    Given the above expression, if we separate $D\times D$ in $\Delta_\eta$ the points at distance at most $\eta$ of the diagonal, and its complement, since $I_{\epsilon}$ is uniformly bounded we can argue that as $\epsilon, \delta$ goes to $0$ the punctual limits of $\widetilde{\mbb P}\left(G_{\epsilon, \alpha}(x),G_{\delta, \alpha}(y)\right)$ and $\Cov(h_{\epsilon}(x), h_{\delta}(y))$ exists (the last one is by hypothesis). On the other hand, by the uniform bound of $I_{\epsilon}$ the integration over $\Delta_{\eta}$ goes to $0$ as $\eta$ goes to 0 uniformly. Given the above, when one do as follows 
    \begin{equation*}
        \E[|I_{\epsilon} - I_{\delta}|^2] = \E[I_{\epsilon}^2 - I_{\epsilon}I_{\delta}] + \E[I_{\delta}^2 - I_{\epsilon}I_{\delta}],
    \end{equation*}
    can conclude as before by arguing via Fatou's lemma. From here we can argue for the existence of the limit in probability as in Section 6 of the same cited work.
\end{proof}

We now prove that the conditions from Proposition \ref{p.mainber} implies the conditions from Theorem \ref{t.generalgmc}. 

\begin{lemma}
    \label{l.priorstogmc} Let $h$ be the Gaussian random field with covariance $K$ given by \eqref{eq.generalKernel}, and assume that it satisfies hypothesis \hyperref[hyp.I]{(H1)} and \hyperref[hyp.II]{(H2)}. Then, the second moment of $I_{\epsilon}$ is uniformly bounded for $\g < \sqrt{2\rho}$ and $\Cov(h_\epsilon(x),h_{\delta}(y))$ has a punctual limit as $x\not =y$.
\end{lemma}
\begin{proof}
It is direct from \hyperref[hyp.I]{(H1)} that there is a punctual limit of $\Cov(h_\epsilon(x),h_{\delta}(y))$ as $x\not =y$.
To prove that $I_{\epsilon}$ is uniformly bounded in $L^2$, we use \eqref{eq.generalcorr}, in particular we first claim that $\widetilde{\mbb P}$ from \eqref{eq.generalchangemeas} fulfils that there is a constant $C>0$ such that uniformly on $x,y,\epsilon$
\begin{equation}
    \label{eq.001}
    \widetilde{\mbb P}\left(G_{\epsilon,\alpha}(x),G_{\epsilon,\alpha}(y)\right)
    \leq C(d(x,y)\lor\epsilon)^{\frac{(2\g - \alpha)^2}{2}}.
\end{equation}
We will not prove the above here since it is the same computation done in Lemma 3.6 from \cite{Ber}. Given the above estimate, using \eqref{eq.001} and \hyperref[hyp.I]{(H1)} in \eqref{eq.generalcorr} we obtain that
\begin{equation*}
    \E\left[I_{\epsilon}^2\right] \leq C
    \int_D\int_D \left(d(x,y)\lor \epsilon\right)^{\frac{(2\g - \alpha)^2}{2} - \g^2}\sigma(dx)\sigma(dy).
\end{equation*}
Here, we apply \hyperref[hyp.II]{(H2)} to notice that the above is bounded if
$\frac{(2\g - \alpha)^2}{2} - \g^2 > -\rho $, and for $\alpha$ near $\g$ this is equivalent to $\g < \sqrt{2\rho}$.
\end{proof}

From the above we proceed to prove Theorem \ref{t.generalgmc}.

\begin{proof}[Proof of Theorem \ref{t.generalgmc}]
    It is the same proof as that given in Section 6 from \cite{Ber}
\end{proof}
\bibliographystyle{alpha}
\bibliography{biblio}
\end{document}